\documentclass[a4paper,11pt,reqno]{smfart}
\usepackage{amssymb,amsmath,mathrsfs,graphics,graphicx,mathptm,float,lscape}
\usepackage[T1]{fontenc}
\usepackage[english, francais]{babel}
\usepackage[all]{xy}
\theoremstyle{plain}
\newtheorem{thm}{Th\'eor\`eme}[section]
\newtheorem{pro}[thm]{Proposition}
\newtheorem{lem}[thm]{Lemme}
\newtheorem{cor}[thm]{Corollaire}
\newtheorem{theoalph}{Theor\`eme}

\newtheorem{coralph}[theoalph]{Corollaire}

\theoremstyle{definition}

\newtheorem{eg}[thm]{Exemple}

\newtheorem{rem}[thm]{Remarque}
\newtheorem{rems}[thm]{Remarques}

\usepackage[colorlinks, linktocpage, 
citecolor = blue, linkcolor = blue]{hyperref}

\def\og{\leavevmode\raise.3ex\hbox{$\scriptscriptstyle\langle\!\langle$~}}
\def\fg{\leavevmode\raise.3ex\hbox{~$\!\scriptscriptstyle\,\rangle\!\rangle$}}

\addtocounter{section}{0}             % Start with section 1
\numberwithin{equation}{section}       % Number formulas within sections

\begin{document}
\selectlanguage{french}
\title{Centralisateurs dans le groupe de Jonqui\`eres}

\author{Dominique \textsc{Cerveau}}

\address{Membre de l'Institut Universitaire de France.
IRMAR, UMR 6625 du CNRS, Universit\'e de Rennes $1$, $35042$ Rennes, France.}
\email{dominique.cerveau@univ-rennes1.fr}

\author{Julie \textsc{D\'eserti}}
\address{Universit\"{a}t Basel, Mathematisches Institut, Rheinsprung $21$, CH-$4051$ Basel, Switzerland.}
\email{julie.deserti@unibas.ch}
\medskip
\address{On leave from Institut de Math\'ematiques de Jussieu, Universit\'e Paris $7$, Projet G\'eom\'etrie et Dynamique, Site Chevaleret, Case $7012$, $75205$ Paris Cedex 13, France.}
\email{deserti@math.jussieu.fr}
\thanks{Le second auteur est soutenu par le Swiss National Science Foundation grant no PP00P2\_128422 /1.}

\maketitle

\begin{altabstract}
\selectlanguage{english}
We give a criterion to determine when the degree growth of a birational map of the complex projective plane which fixes (the action on the basis of the fibration is trivial) a rational fibration is linear up to conjugacy. We also compute the centraliser of such maps. It allows us to describe the centraliser of the birational maps of the complex projective plane which preserve a rational fibration (the action on the basis of the fibration being not necessarily trivial); this question is related to some classical problems of difference equations. 

\noindent\emph{$2010$ Mathematics Subject Classification. --- $14E05$, $14E07$}
\end{altabstract}

\selectlanguage{french}

\section{Introduction}

\noindent La description des centralisateurs des syst\`emes dynamiques discrets est consid\'er\'ee comme un probl\`eme important tant en dynamique r\'eelle que complexe. 
Ainsi Julia (\cite{Julia, Julia2}) puis Ritt (\cite{Ritt}) montrent que l'ensemble $\mathrm{C}(f)=\big\{g\colon\mathbb{P}^1\to\mathbb{P}^1\,\big\vert\, fg=gf\big\}$ des 
fonctions rationnelles commutant \`a une fonction rationnelle fix\'ee $f$ se r\'eduit en g\'en\'eral aux it\'er\'es $f_0^\mathbb{N}=\big\{f_0^n\,\vert\,n\in\mathbb{N}\big\}$ 
de l'un de ses \'el\'ements sauf dans quelques cas sp\'eciaux (\`a conjugaison pr\`es les mon\^omes $z^k$, les polyn\^omes de Tchebychev, les exemples de Latt\`es...) Plus tard, 
dans les ann\'ees $60$, Smale demande si un diff\'eomorphisme g\'en\'erique $f\colon M\to M$ d'une vari\'et\'e compacte a son centralisateur trivial, c.-\`a.d. si 
$$\mathrm{C}(f)=\big\{g\in\mathrm{Diff}^\infty(M)\,\vert\, fg=gf\big\}$$ est r\'eduit \`a $f^\mathbb{Z}=\big\{f^n\,\vert\,n\in\mathbb{Z}\big\}$. De nombreux math\'ematiciens 
ont apport\'e leur contribution parmi lesquels Bonatti, Crovisier, Fisher, Palis, Wilkinson, Yoccoz (\cite{BCW, Fi, Fi2, Pa, PaYo2, PaYo}). De m\^eme dans le cadre local, 
plus pr\'ecis\'ement dans l'\'etude des \'el\'ements de $\mathrm{Diff}(\mathbb{C},0)$, le groupe des germes de diff\'eomorphismes holomorphes \`a l'origine de $\mathbb{C}$, 
la description des centralisateurs joue un r\^ole crucial. Ainsi Ecalle montre que si~$f\in\mathrm{Diff}(\mathbb{C},0)$ est tangent \`a l'identit\'e alors, sauf cas exceptionnel, 
son centralisateur se r\'eduit encore \`a un $f_0^\mathbb{Z}$ (\emph{voir} \cite{Ec, Ec2}); ceci permet en particulier de d\'ecrire les sous-groupes r\'esolubles non ab\'eliens de
 $\mathrm{Diff}(\mathbb{C},0)$ (\emph{voir} \cite{CM}). \`A l'inverse Perez-Marco montre qu'il existe des sous-groupes ab\'eliens non lin\'earisables et non d\'enombrables 
de~$\mathrm{Diff}(\mathbb{C},0)$ en liaison avec des questions difficiles de \og mauvais petits diviseurs\fg\, (\cite{PM}).

\bigskip

\noindent Nous nous proposons de terminer (\`a indice fini pr\`es) le calcul des centralisateurs des \'el\'ements du groupe de Cremona $\mathrm{Bir}(\mathbb{P}^2)$. Comme 
on peut le voir entre autres dans~\cite{BlDe, CaCe} ces calculs de centralisateurs apparaissent naturellement dans les probl\`emes de repr\'esentations de groupes abstraits 
dans les groupes de transformations. Le groupe $\mathrm{Bir}(\mathbb{P}^2)$ des transformations birationnelles du plan projectif complexe est constitu\'e des applications 
rationnelles
\begin{align*}
&\phi\colon\mathbb{P}^2(\mathbb{C})\dashrightarrow\mathbb{P}^2(\mathbb{C}), && (x:y:z)\dashrightarrow(\phi_0(x,y,z):\phi_1(x,y,z):\phi_2(x,y,z)),
\end{align*}
les $\phi_i$ d\'esignant des polyn\^omes homog\`enes de m\^eme degr\'e et sans facteur commun, qui admettent (sur un ouvert de Zariski) un inverse du m\^eme type. Rappelons 
que le premier degr\'e dynamique de $\phi$ est donn\'e par~$\lambda(\phi)=\displaystyle\lim_n(\deg \phi^n)^{1/n}$. 

\noindent Par \og homog\'en\'eisation\fg\, chaque automorphisme polynomial de $\mathbb{C}^2$ produit un \'el\'ement de $\mathrm{Bir}(\mathbb{P}^2)$, ainsi le 
{\it grou\-pe}~$\mathrm{Aut}[\mathbb{C}^2]$ {des automorphismes polynomiaux de} $\mathbb{C}^2$ peut \^etre vu comme un sous-groupe de $\mathrm{Bir}(\mathbb{P}^2)$; c'est le 
produit amalgam\'e du groupe~$\mathrm{A}$ des automorphismes affines et du {\it groupe \'el\'ementaire} d\'efini par
$$\mathrm{E}=\big\{(\alpha x+P(y),\beta y+\gamma)\,\big\vert\,\alpha,\,\beta,\,\gamma\in\mathbb{C},\,\alpha\beta\not =0,\, P\in\mathbb{C}[y]\big\}$$
le long de leur intersection (\cite{Ju}). En s'appuyant sur cette structure tr\`es particuli\`ere Friedland et Milnor ont montr\'e l'alternative suivante (\cite{FM}): soit $\phi$ 
dans $\mathrm{Aut} [\mathbb{C}^2]$ alors
\begin{itemize}
\item ou bien $\phi$ est conjugu\'e \`a un \'el\'ement de $\mathrm{E}$; 

\item ou bien $\phi$ est de {\it type H\'enon}, c.-\`a.d. conjugu\'e \`a $g_1\ldots g_n$ o\`u $g_i=(y,P_i(y)-\delta_i x)$, $P_i\in\mathbb{C}[y]$, $\deg P_i\geq 2$, $\delta_i\in
\mathbb{C}^*$.
\end{itemize} 
Si $\phi$ est dans $\mathrm{Aut} [\mathbb{C}^2]$, on a $\lambda(\phi)=1$ (resp. $\lambda(\phi)>~1$) si et seulement si $\phi$ est un automorphisme \'el\'ementaire (resp. un 
automorphisme de type H\'enon). \'Etant donn\'e un automorphisme polynomial $\phi$ de $\mathbb{C}^2,$ on appelle centralisateur de $\phi$ le groupe des automorphismes polynomiaux 
qui commutent \`a $\phi$. Mis \`a part les automorphismes r\'esonants du type 
\begin{align*}
&(\beta^dx+\beta^dy^dq(y^r),\beta y), &&\beta^r=1,\,d\geq 1, q\in\mathbb{C}[y]\setminus\mathbb{C}
\end{align*}
tout \'el\'ement de $\mathrm{E}$ se plonge dans un flot. En particulier le centralisateur d'un \'el\'ement de $\mathrm{E}$ non r\'esonant est non d\'enombrable; quant aux 
\'el\'ements r\'esonants ils commutent aux automorphismes de la forme $(x+ay^d,y)$, $a\in\mathbb{C}$, leur centralisateur est donc aussi non d\'enombrable. Par contre le 
centralisateur d'un \'el\'ement de type H\'enon est d\'enombrable; plus pr\'ecis\'ement si $\phi$ est type H\'enon, son centralisateur est isomorphe \`a $\mathbb{Z}\rtimes
\mathbb{Z}/p\mathbb{Z}$ pour un certain $p$ et g\'en\'erique\-ment~$p=1$ (\emph{voir} \cite{La}). On a donc pour tout $\phi$ dans $\mathrm{Aut}[\mathbb{C}^2]$ l'\'equivalence 
suivante: $\phi$ est de type H\'enon si et seulement si son centralisateur est d\'enombrable. Le groupe~$\mathrm{Aut}[\mathbb{C}^2]$ pr\'esente donc une dichotomie: d'un 
c\^ot\'e les automorphismes de type H\'enon, c\'el\`ebres pour leur dynamique compliqu\'ee, et de l'autre les automorphismes conjugu\'es \`a des \'el\'ements de $\mathrm{E}$. 
Cette dichotomie peut se ca\-ract\'eriser via le degr\'e dynamique, ou via le caract\`ere d\'enombrable ou non du centralisateur, ou encore via la pr\'esence d'une fibration 
rationnelle invariante. Comme nous allons le voir cette dichotomie ne persiste pas dans $\mathrm{Bir}(\mathbb{P}^2)$. Toutefois il y a un analogue \`a l'\'enonc\'e de Friedland 
et Milnor (\cite{DiFa, Gi}); une transformation birationnelle $\phi$ de $\mathbb{P}^2(\mathbb{C})$ satisfait une et une seule des propri\'et\'es suivantes:
\smallskip
\begin{itemize}
\item La suite $(\deg \phi^n)_n$ est born\'ee, $\phi$ est birationnellement conjugu\'e \`a un automorphisme d'une surface rationnelle et un it\'er\'e de $\phi$ est isotope 
\`a l'identit\'e.

\noindent Dans ce cas $\phi$ pr\'eserve une infinit\'e de fibrations rationnelles.

\smallskip

\item $\deg \phi^n\sim n$ alors $\phi$ n'est pas conjugu\'e \`a un automorphisme d'une surface rationnelle et $\phi$ pr\'eserve une unique fibration qui est rationnelle. \smallskip

\item $\deg \phi^n\sim n^2$ auquel cas $\phi$ est conjugu\'e \`a un automorphisme d'une surface rationnelle pr\'eservant une unique fibration qui est elliptique.\smallskip

\item $\deg \phi^n\sim\alpha^n$, $\alpha>1$.\smallskip
\end{itemize}

\noindent Dans les trois premi\`eres \'eventualit\'es on a $\lambda(\phi)=1$, dans la derni\`ere $\alpha=\lambda(\phi)>1$.

\noindent \`A l'inverse du cas polynomial il n'y a pas d'\'equivalence entre le fait d'avoir un centralisateur d\'enombrable et d'avoir un premier degr\'e dynamique 
strictement sup\'erieur \`a $1$. Dans la suite si $\phi$ est un \'el\'ement de $\mathrm{Bir}(\mathbb{P}^2)$ on note $\mathrm{C}(\phi)$ le centralisateur de $\phi$ 
dans $\mathrm{Bir}(\mathbb{P}^2)$: $$\mathrm{C}(\phi)=\big\{\psi\in\mathrm{Bir}(\mathbb{P}^2)\,\vert\,\psi\phi=\phi\psi\big\}.$$

\noindent La description des centralisateurs des \'el\'ements du groupe de Cremona d'ordre infini et \`a croissance born\'ee figure dans \cite{BlDe}. Elle d\'ecoule du 
fait suivant (\cite[Proposition 2.3]{BlDe}): soit~$\phi$ une transformation birationnelle d'ordre infini et telle que $(\deg\phi^n)_n$ soit born\'ee. Alors $\phi$ est conjugu\'e:
\begin{itemize}
\item
soit \`a $(\alpha x,\beta y)$, avec  $\alpha$, $\beta\in \mathbb{C}^*$, le noyau du morphisme $\mathbb{Z}^2 \to \mathbb{C}^*$, $(i,j)\mapsto \alpha^i \beta^j$ \'etant 
engendr\'e par $(k,0)$  pour un certain $k\in \mathbb{Z}$;
 
 \item soit \`a $(\alpha x, y+1)$, avec $\alpha\in \mathbb{C}^*$.
\end{itemize}
Ensuite on montre, dans la premi\`ere \'eventualit\'e, que (\cite{BlDe}) $$\mathrm{C}(\phi)=\big\{(\eta (x),y R(x^k))\ \big\vert\ R\in \mathbb{C}(x), 
\eta \in \mathrm{PGL}_2(\mathbb{C}), \eta(\alpha x)=\alpha\eta(x)\big\};$$ dans le cas g\'en\'erique ($k=0$), on a $\mathrm{C}(\phi)=\big\{(\gamma x,\delta y)\, 
\big\vert\, \gamma,\,\delta\in\mathbb{C}^*\big\}.$
Enfin dans le second on obtient (\cite{BlDe}) $$\mathrm{C}(\phi)=\big\{(\eta(x),y+R(x))\,\big\vert\,\eta\in \mathrm{PGL}_2(\mathbb{C}), \eta(\alpha x)=
\alpha \eta(x), R\in \mathbb{C}(x), R(\alpha x)=R(x)\big\}.$$ Notons que dans ce dernier cas pour $\alpha$ g\'en\'erique, $\mathrm{C}(\phi)$ est isomorphe \`a 
$\mathbb{C}^*\times\mathbb{C}$.

\noindent Dans le cas des transformations birationnelles d'ordre fini, la situation est tr\`es diff\'erente du cas lin\'eaire. Alors que le centralisateur d'une 
involution lin\'eaire (par exemple $(-x,y)$) est infini non d\'enombrable celui d'une involution de Bertini, resp. Geiser, est fini (\cite{BPV}). Pour des involutions 
de Bertini et Geiser g\'en\'eriques $\mathcal{I}$ le centralisateur se r\'eduit \`a $\big\{\mathrm{id},\,\mathcal{I}\big\}$ mais il existe de telles involutions avec 
un centralisateur plus gros. En effet dans \cite[Proposition 3.12]{CeDe} on construit \`a partir d'un feuilletage $\mathcal{F}_J$, dit de Jouanolou, de degr\'e~$2$ sur 
$\mathbb{P}^2(\mathbb{C})$ une involution birationnelle $\mathcal{I}_{\mathcal{F}_J}$; on montre que $\mathcal{I}_{\mathcal{F}_J}$ est une involution de Geiser. On 
constate que le groupe engendr\'e par $\mathcal{I}_{\mathcal{F}_J}$ et le groupe d'isotropie de $\mathcal{F}_J$ est un groupe fini d'ordre $42$ qui est contenu 
dans~$\mathrm{C}(\mathcal{I}_{\mathcal{F}_J})$.

\noindent Notons que lorsque $\deg\phi^n\sim n^2$ la description des automorphismes des courbes elliptiques permet d'affirmer que $\mathrm{C}(\phi)$ est virtuellement 
ab\'elien tout du moins dans le cas g\'en\'erique. En utilisant la classification des pinceaux de courbes elliptiques (pinceaux d'Halphen) on peut pr\'eciser dans tous 
les cas ces $\mathrm{C}(\phi)$ (\emph{voir} \cite{Ca})

\noindent Dans \cite{Ca} Cantat donne une description du centralisateur des \'el\'ements de premier degr\'e dynamique strictement sup\'erieur \`a $1$: soit $\phi$ une 
transformation birationnelle du plan projectif
complexe dont le premier degr\'e dynamique $\lambda(\phi)$ est strictement plus grand que $1$. Si $\psi$ est un \'el\'ement de $\mathrm{Bir}(\mathbb{P}^2)$ qui commute 
avec $\phi$, il existe deux entiers $m$ dans~$\mathbb{N}^*$ et $n$ dans $\mathbb{Z}$ tels que~$\psi^m=\phi^n$. 

\noindent  Dans \cite{De2} on \'etudie une famille de transformations birationnelles~$(f_{\alpha,\beta})$ ayant les propri\'et\'es suivantes: $f_{\alpha,\beta}$ est 
\`a croissance lin\'eaire, en particulier~$\lambda(f_{\alpha,\beta})=1,$ et poss\`ede un centralisateur d\'enombrable: $\mathrm{C}(f_{\alpha,\beta})=\langle f_{\alpha,
\beta}^n\,\vert\,n\in\mathbb{Z}\rangle$.

\noindent Dans cet article on s'attache en particulier \`a compl\'eter le cas manquant, celui des transformations birationnelles \`a croissance lin\'eaire. 
 Ceci termine tout du moins du point de vue qualitatif l'\'etude des centralisateurs des transformations birationnelles qui ne sont pas de torsion.

\bigskip

\noindent Dans un premier temps on \'etablit un crit\`ere \'el\'ementaire permettant d'affirmer qu'une transformation birationnelle  qui pr\'eserve une fibration rationnelle fibre \`a 
fibre est \`a croissance lin\'eaire (\S\ref{critere}); ce crit\`ere met en jeu un invariant matriciel projectif classique de type \og Baum-Bott\fg.

\begin{theoalph}\label{pr}
Soit $\phi$ un \'el\'ement du groupe de Cremona qui pr\'eserve une fibration rationnelle fibre \`a fibre. La transformation $\phi$ est \`a croissance lin\'eaire si 
et seulement si l'indice de Baum-Bott de $\phi$ appartient \`a~$\mathbb{C}(y)\setminus\mathbb{C}$.
\end{theoalph}

\noindent Une cons\'equence de \cite[Theorem 0.2]{DiFa} et du Th\'eor\`eme \ref{pr} est le fait suivant: {\it soit $\phi$ un \'el\'ement du groupe de Cremona qui
 pr\'eserve une fibration rationnelle fibre \`a fibre. La transformation $\phi$ est \`a croissance born\'ee si et seulement si l'indice de Baum-Bott de $\phi$ 
appartient \`a $\mathbb{C}$}.

\medskip

\noindent Au \S\ref{calculcentralisateur} on commence par \'etablir la propri\'et\'e suivante: {\it 
soit $\phi$ une transformation birationnelle qui pr\'eserve une fibration rationnelle $\mathcal{F}$ fibre \`a fibre. Alors ou bien $\phi$ est p\'eriodique, ou bien 
tous les \'el\'ements qui commutent \`a $\phi$ laissent $\mathcal{F}$ invariante $($pas n\'ecessairement fibre \`a fibre$)$.}

\noindent Soit $\phi$ une transformation birationnelle qui pr\'eserve une fibration rationnelle $\mathcal{F}$ fibre 
\`a fibre; $\phi$ est contenu dans un sous-groupe ab\'elien maximal, not\'e $\mathrm{Ab}(\phi)$, pr\'eservant 
$\mathcal{F}$ fibre \`a fibre. La nature de ces sous-groupes ab\'eliens maximaux est bien connue, nous la rappelons 
au \S\ref{Sec:J0}. On d\'ecrit les centralisateurs $\mathrm{C}(\phi)$ des \'el\'ements non p\'eriodiques $\phi$ de 
$\mathrm{Bir}(\mathbb{P}^2)$ qui pr\'eservent une fibration rationnelle fibre \`a fibre (\S\ref{calculcentralisateur}); 
cette description conduit \`a l'\'enonc\'e suivant.

\begin{theoalph}\label{ababelien}
Soit $\phi$ une transformation birationnelle qui pr\'eserve une fibration rationnelle fibre \`a fibre. Si $\phi$ est 
\`a croissance lin\'eaire, le centralisateur $\mathrm{C}(\phi)$ de $\phi$ est une extension finie de $\mathrm{Ab}(\phi)$.
\end{theoalph}

\noindent Ces r\'esultats nous permettent de d\'ecrire, \`a indice fini pr\`es, les centralisateurs des \'el\'ements 
$\phi$ qui pr\'eservent une unique fibration rationnelle, pas n\'ecessairement fibre \`a fibre, question en liaison 
avec des probl\`emes classiques d'\'equations aux diff\'erences (\S\ref{Sec:centralisateursj}). Nous verrons que 
g\'en\'eriquement ces transformations ont un centralisateur trivial, c.-\`a.d. r\'eduit aux it\'er\'es de~$\phi$. 
Ceci est r\'esum\'e dans un tableau r\'ecapitulatif au \S\ref{Sec:res}. Un corollaire imm\'ediat de cette description 
est l'\'enonc\'e suivant:

\begin{coralph}
Soit $\phi$ une transformation birationnelle du plan projectif complexe qui pr\'eserve une unique fibration
rationnelle. Le centralisateur de $\phi$ est virtuellement r\'esoluble.
\end{coralph}

\subsection*{Remerciements} Merci \`a J.-J. Loeb pour sa participation sympathique et au rapporteur pour ses 
remarques et propos judicieux.

\section{Groupe de Jonqui\`eres}\label{jojo}

\subsection{Quelques rappels sur le groupe de Cremona} 
%Consid\'erons le groupe des transformations birationnelles du plan projectif complexe, encore appel\'e {\it groupe de Cremona}, que l'on note $\mathrm{Bir}(\mathbb{P}^2)$. 
%Ce groupe est constitu\'e des applications rationnelles
%\begin{align*}
%&\phi\colon\mathbb{P}^2(\mathbb{C})\dashrightarrow\mathbb{P}^2(\mathbb{C}), && (x:y:z)\dashrightarrow(\phi_0(x,y,z):\phi_1(x,y,z):\phi_2(x,y,z)),
%\end{align*}
%les $\phi_i$ d\'esignant des polyn\^omes homog\`enes de m\^eme degr\'e et sans facteur commun, qui admettent un inverse du m\^eme type (sur un ouvert de Zariski). 
%Le {\it degr\'e} d'une transformation $\phi$ est par d\'efinition le degr\'e des~$\phi_i$; il se trouve que ce n'est pas un invariant de conjugaison (pour $f$, $g\in
%\mathrm{Bir}(\mathbb{P}^2)$ g\'en\'eriques on a $\deg fgf^{-1}\not=\deg g$). On introduit l'invariant de conjugaison suivant: $\lambda(\phi)=\displaystyle\lim_n(\deg 
%\phi^n)^{1/n}$ que l'on appelle {\it premier degr\'e dynamique} de $\phi$. 

\noindent Soit $\phi$ un \'el\'ement de $\mathrm{Bir}(\mathbb{P}^2)$. L'ensemble des {\it points d'ind\'etermination} de $\phi$ est par d\'efinition l'ensemble $$\mathrm{Ind}\,
\phi=\big\{m\in\mathbb{P}^2(\mathbb{C})\,\vert\,\phi_0(m)=\phi_1(m)=\phi_2(m)=0\big\};$$ quant \`a l'{\it ensemble exceptionnel} de $\phi$, ou encore l'ensemble des courbes 
contract\'ees par $\phi$, il est donn\'e par $$\mathrm{Exc}\,\phi=\big\{m\in\mathbb{P}^2(\mathbb{C})\,\vert\,\det\mathrm{jac}\phi_{(m)}=0\big\}.$$ On note que la transformation 
birationnelle $\phi$ induit un automorphisme de $\mathbb{P}^2(\mathbb{C})$ si et seulement si $\mathrm{Ind}\,\phi=\mathrm{Exc}\,\phi=\emptyset$. Pour les transformations 
$f_{\alpha,\beta}$ \'evoqu\'ees dans l'introduction et donn\'ees par
\begin{align*} 
&f_{\alpha,\beta}=(z(\alpha x+y):\beta y(x+z):z(x+z)), && \alpha,\,\beta\in\mathbb{C}^*.
\end{align*} 
on a 
\begin{align*}
&\mathrm{Ind}\,f_{\alpha,\beta}=\big\{(1:0:0),\,(0:1:0),\,(-1:\alpha:1)\big\}, && \mathrm{Exc}\,f_{\alpha,\beta}=\big\{z=-x\big\}\cup\big\{z=0\big\}\cup\big\{y=\alpha z\big\}.
\end{align*}

\subsection{Groupe de Jonqui\`eres} Le {\it groupe de Jonqui\`eres} est le groupe des transformations birationnelles du plan projectif complexe pr\'eservant un pinceau 
de courbes rationnelles; on le note $\mathrm{J}$. Comme deux pinceaux de courbes rationnelles sont birationnellement conjugu\'es, $\mathrm{J}$ ne d\'epend pas, \`a conjugaison 
pr\`es, du pinceau choisi. Dit autrement on peut supposer \`a conjugaison birationnelle pr\`es que $\mathrm{J}$ est, dans une carte affine $(x,y)$ de $\mathbb{P}^2(\mathbb{C})$, 
le groupe maximal des transformations birationnelles laissant la fibration $y=$ cte invariante. Une transformation $\phi$ de $\mathrm{J}$
permutant les fibres de la fibration,~$\phi$ induit un automorphisme de la base $\mathbb{P}^1(\mathbb{C})$, soit un \'el\'ement de $\mathrm{PGL}_2(\mathbb{C})$; lorsque 
$\phi$ pr\'eserve les fibres, $\phi$ agit comme une homographie dans les fibres g\'en\'eriques. Le groupe de Jonqui\`eres s'identifie donc au produit semi-direct 
$\mathrm{PGL}_2(\mathbb{C}(y))\rtimes\mathrm{PGL}_2(\mathbb{C})$; plus pr\'ecis\'ement tout $\phi$ dans~$\mathrm{J}$ s'\'ecrit 
\begin{align*}
& \phi=\left(\frac{A(y)x+B(y)}{C(y)x+D(y)},\frac{\alpha y+\beta}{\gamma y+\delta}\right), && \left[\begin{array}{cc} A & B\\ C & D\end{array}\right]\in\mathrm{PGL}_2(
\mathbb{C}[y]), && \left[\begin{array}{cc} \alpha & \beta \\ \gamma & \delta\end{array}\right]\in\mathrm{PGL}_2(\mathbb{C}).
\end{align*}
On remarque que pour un tel $\phi$ l'ensemble exceptionnel se r\'eduit \`a un nombre fini de fibres $y=$ cte et \'eventuellement la \og droite\fg\, \`a l'infini.
On d\'esigne par $\pi_2$ le morphisme de $\mathrm{J}$ dans $\mathrm{PGL}_2(\mathbb{C})$, c.-\`a.d. $\pi_2(\phi)$ est la seconde composante de $\phi$. Les \'el\'ements 
de~$\mathrm{J}$ qui pr\'eservent la fibration fibre \`a fibre forment un sous-groupe distingu\'e (noyau du morphisme $\pi_2$) que nous note\-rons~$\mathrm{J}_0\simeq
\mathrm{PGL}_2(\mathbb{C}(y))$. Si $\phi$ appartient \`a~$\mathrm{J}_0$, il s'\'ecrit $\left(\frac{A(y)x+B(y)}{C(y)x+D(y)},y\right)$ et on note $\mathrm{M}_\phi$ la 
matrice associ\'ee $\left[\begin{array}{cc} A & B \\ C & D \end{array}\right]$. L'{\it indice de Baum-Bott} (par analogie avec l'indice de Baum-Bott d'un feuilletage, 
\cite{BB, GML}) de $\phi$ est par d\'efinition $\frac{(\mathrm{tr}\,\mathrm{M}_\phi)^2}{\det\mathrm{M}_\phi}$, on le notera $\mathrm{BB}(\phi)$. Le groupe~$\mathrm{J}$
 contient bien \'evidemment des transformations \`a croissance born\'ee; on va donner un crit\`ere permettant de d\'eterminer les \'el\'ements de $\mathrm{J}_0$ qui sont 
\`a croissance lin\'eaire (\S\ref{critere}).

\subsection{Formes normales dans $\mathrm{J}_0$}\label{Sec:J0}

\noindent Soit $\phi$ une transformation birationnelle appartenant \`a $\mathrm{J}_0$; elle est, \`a conjugaison birationnelle pr\`es, de l'un des types suivants (\emph{voir} 
par exemple \cite{De2})
\begin{align*}
&\mathfrak{a}\,\,\,\,\,(x+a(y),y), && \mathfrak{b}\,\,\,\,\,(b(y)x,y),&& \mathfrak{c}\,\,\,\,\,\left(\frac{c(y)x+F(y)}{x+c(y)},y\right), 
\end{align*}

\noindent avec $a$ dans $\mathbb{C}(y)$, $b$ dans $\mathbb{C}(y)^*$ et $c$, $F$ dans $\mathbb{C}[y]$, $F$ n'\'etant pas un carr\'e (si $F$ est un carr\'e alors $\phi$ 
est conjugu\'e \`a un \'el\'ement de type $\mathfrak{b}$). En effet, soient $\mathrm{M}_\phi=\left[\begin{array}{cc}A &B \\ C & D\end{array}\right]$ un \'el\'ement de 
$\mathrm{PGL}_2( \mathbb{C}[y])$ et $$P_{\mathrm{M}_\phi}(X)=X^2 -(\mathrm{tr}\,\mathrm{M}_\phi) X+ \det\mathrm{M}_\phi$$ son polyn\^ome caract\'eristique. Ou bien 
$P_{\mathrm{M}_\phi}$ a deux racines distinctes (cas~$\mathfrak{b}$), ou bien $P_{\mathrm{M}_\phi}$ a une racine double (cas $\mathfrak{a}$), ou bien~$P_{\mathrm{M}_\phi}$ 
n'a pas de racine dans $\mathbb{C}[y]$. Dans ce dernier cas ceci signifie que $\mathrm{tr}^2\mathrm{M}_\phi-4\det\mathrm{M}_\phi$ n'est pas un carr\'e dans~$\mathbb{C}[y]$; 
on en d\'eduit que $BC\not= 0$ et que $B(y)$ s'\'ecrit $G(y)-\frac{(A(y)-D(y))^2}{4}$ o\`u $G$ d\'esigne un \'el\'ement de $\mathbb{C}[y]$ qui n'est pas un carr\'e. Quitte
 \`a conjuguer $\phi$ par $(C(y)x,y)$, on peut supposer que $C= 1$ et que~$A$, $B$, $D$ appartiennent \`a $\mathbb{C}[y]$. Ensuite en conjuguant $\phi$ par $\left(x+
\frac{D(y)-A(y)}{2}, y\right)$ on se ram\`ene \`a $A= D$ et $B$ n'est pas un carr\'e; autrement dit~$\mathrm{M}_\phi=\left[\begin{array}{cc}A & B\\ 1 & A\end{array}\right]$ 
avec~$A$, $B$ dans~$\mathbb{C}[y]$ et $B$ n'est pas un carr\'e. 

\noindent Les sous-groupes ab\'eliens maximaux non finis de $\mathrm{J}_0$ sont (\`a conjugaison pr\`es dans $\mathrm{J}_0$) de trois types
\begin{align*}
&\mathrm{J}_a=\big\{(x+a(y),y)\,\big\vert\, a\in\mathbb{C}(y)\big\}, &&\mathrm{J}_m=\big\{(b(y)x,y)\,\big\vert\, b\in\mathbb{C}(y)^*\big\}
\end{align*}
et les groupes
\begin{align*}
\mathrm{J}_F=\left\{(x,y),\,\left(\frac{c(y)x+F(y)}{x+c(y)},y\right)\,\Big\vert\, c\in\mathbb{C}(y)\right\}
\end{align*}
o\`u $F$ d\'esigne un \'el\'ement de $\mathbb{C}[y]$ qui n'est pas un carr\'e (\cite{De2}).
En fait en conjuguant par une transformation~$(a(y)x,y)$ ad-hoc on peut supposer que $F$ est un polyn\^ome ayant toutes ses racines simples, ce que nous faisons dans la suite. 

\begin{rem}
Un \'el\'ement de type $\mathfrak{a}$ non trivial est conjugu\'e (via une transformation de type $\mathfrak{b}$) \`a $(x+~1,y)$. Les trois types de groupes 
pr\'ec\'edents se distinguent par la nature des points fixes de leurs \'el\'ements; avec des notations \'evidentes $\mathrm{Fix}\,\mathrm{J}_a=\{(x=\infty,y)\}$,
$\mathrm{Fix}\,\mathrm{J}_m=\{(x=\infty,y)\}\cup\{(x=0,y)\}$, $\mathrm{Fix}\,\mathrm{J}_F=\{(x,y)\,\vert\, x^2=F(y)\}$, ces \'egalit\'es \'etant \'ecrites 
dans $\mathbb{P}^1\times\mathbb{P}^1$.
\end{rem}

Par suite si $\phi$ est un \'el\'ement de $\mathrm{J}_0$ et si~$\mathrm{Ab}(\phi)$ d\'esigne le sous-groupe ab\'elien maximal non fini de $\mathrm{J}_0$ contenant $\phi$ 
alors, \`a conjugaison pr\`es, $\mathrm{Ab}(\phi)$ est $\mathrm{J}_a$, $\mathrm{J}_m$ ou $\mathrm{J}_F$. Plus pr\'ecis\'ement si $\phi$ est de type~$\mathfrak{a}$ (resp. 
$\mathfrak{b}$, resp. $\mathfrak{c}$), alors $\mathrm{Ab}(\phi)=\mathrm{J}_a$ (resp.~$\mathrm{Ab}(\phi)=\mathrm{J}_m$, resp. $\mathrm{Ab}(\phi)=\mathrm{J}_F$).

\section{Croissance des degr\'es dans $\mathrm{J}_0$}\label{critere}

\noindent Nous allons dans ce paragraphe donner un crit\`ere permettant d'affirmer qu'un \'el\'ement $\phi$ de $\mathrm{J}_0$ est \`a croissance lin\'eaire.

\medskip

\noindent Comme on l'a rappel\'e pr\'ec\'edemment $\phi$ est, \`a conjugaison birationnelle pr\`es, de l'un des types $\mathfrak{a}$, $\mathfrak{b}$ ou $\mathfrak{c}$. 

\begin{rem}
La transformation $\phi$ est une involution si et seulement si l'indice de Baum-Bott de $\phi$ est nul.
\end{rem}

\begin{rem}
Posons $$\Xi=\Big\{\alpha+\frac{1}{\alpha}+2\,\vert\,\alpha\in\mathbb{C}(y)^*\Big\}.$$
Soient $\phi\in\mathrm{J}_0$ et $\mathrm{BB}(\phi)$ son indice de Baum-Bott. Si $\phi$ n'est pas une
involution, c.-\`a.d. $\mathrm{BB}(\phi)\not=0$, on a
\begin{itemize}
\item $\phi$ est du type $\mathfrak{a}$ si et seulement si $\mathrm{BB}(\phi)=4$;

\item $\phi\not=\mathrm{id}$ est du type $\mathfrak{b}$ si et seulement si $\mathrm{BB}(\phi)$ appartient \`a $\Xi$;

\item $\phi$ est du type $\mathfrak{c}$ si et seulement si $\mathrm{BB}(\phi)$ n'est pas dans $\Xi$.
\end{itemize}
\end{rem}

\smallskip

\noindent Si $\phi$ est de type $\mathfrak{a}$, on a $\deg \phi^n=\deg \phi$. Lorsque $\phi=(a(y)x,y)$ est de type $\mathfrak{b}$, alors
\begin{align*}
&\mathrm{BB}(\phi)=\frac{(a(y)+1)^2}{a(y)},&& \phi^n=(a(y)^nx,y)
\end{align*}

\noindent et on a l'alternative suivante: 
\begin{itemize}
\item ou bien $a$ est une constante auquel cas l'indice de Baum-Bott de $\phi$ est une constante et $\deg \phi^n=\deg \phi=1;$ 

\item ou bien $a$ appartient \`a $\mathbb{C}(y)\setminus\mathbb{C}$ et 
$\mathrm{BB}(\phi)$ appartient \`a $\mathbb{C}(y)\setminus\mathbb{C}$, alors $\deg \phi^n=n\deg  \phi-n+1$.
\end{itemize}

\noindent Toute transformation $\phi=\Big(\frac{c(y)x+F(y)}{x+c(y)},y\Big)$ de type $\mathfrak{c}$, c.-\`a.d. appartenant \`a un $\mathrm{J}_F$, est 
conjugu\'ee dans $\mathrm{PGL}_2\big(\mathbb{C}(y)
+\sqrt{F}\mathbb{C}(y)\big)$ \`a $\psi=\left(\frac{c(y)+\sqrt{F(y)}}{c(y)-\sqrt{F(y)}}x,y\right)$. Puisque~$\mathbb{C}(y)+\sqrt{F}\mathbb{C}(y)$ 
s'identifie au corps des fonctions m\'eromorphes sur la courbe $\mathcal{C}=\{x^2=F(y)\}$, il y a sur ce corps une notion de degr\'e (nombre de 
points dans la fibre g\'en\'erique). Par exemple le degr\'e de $F$ (c.-\`a.d. le degr\'e de $x_{\vert\mathcal{C}}$) est le degr\'e ordinaire de 
$F$, celui de $y$ est $2$. On v\'erifie que la croissance des degr\'es de $\phi$ et $\psi$ sont identiques. La transformation $\psi$ \'etant du type 
$\mathfrak{b}$ sur $\mathcal{C}$, elle est \`a croissance born\'ee si et seulement si $\frac{c(y)+\sqrt{F(y)}}{c(y)-\sqrt{F(y)}}$ est une constante autrement dit si et seulement 
si $c$ est nul (puisque $F$ n'est pas un carr\'e). Par ailleurs $\mathrm{BB}(\phi)=\frac{4c(y)^2}{c(y)^2-F(y)}$; il appartient \`a $\mathbb{C}(y)\setminus\mathbb{C}$ 
si et seulement si $c$ est non nul (toujours parce que $F$ n'est pas un carr\'e).

\medskip

\noindent Ce qui pr\'ec\`ede nous permet d'\'enoncer le:

\begin{thm}\label{cri}
Soit $\phi$ un \'el\'ement de $\mathrm{J}_0$. La transformation $\phi$ est \`a croissance lin\'eaire si et seulement si l'indice de Baum-Bott de $\phi$ appartient 
\`a $\mathbb{C}(y)\setminus\mathbb{C}$.
\end{thm}

\begin{rem}
Ceci permet d'obtenir un crit\`ere pour les \'el\'ements $\phi$ tels que $\pi_2(\phi)$ soit p\'eriodique de p\'eriode $k$: un tel $\phi$ est \`a croissance 
lin\'eaire  (resp. born\'ee) si et seulement si $\mathrm{BB}(\phi^k)$ appartient \`a $\mathbb{C}(y)\setminus\mathbb{C}$ (resp. $\mathrm{BB}(\phi^k)$ appartient 
\`a~$\mathbb{C}$).
\end{rem}

\begin{cor}
Soit $\phi$ un \'el\'ement de $\mathrm{J}_0$. La transformation $\phi$ est \`a croissance born\'ee si et seulement si l'indice de Baum-Bott de $\phi$ appartient 
\`a $\mathbb{C}$; dans ce cas $\phi$ est soit de type $\mathfrak{a}$, soit de type $\mathfrak{b}$ lin\'eaire $((bx,y)$ avec $b$ dans $\mathbb{C}^*)$, soit une
 involution $\left(\frac{F(y)}{x},y\right)$ de type $\mathfrak{c}$.
\end{cor}

\section{Centralisateurs des \'el\'ements de $\mathrm{J}_0$}\label{calculcentralisateur}

\subsection{Le centralisateur d'une transformation non p\'eriodique de $\mathrm{J}_0$ est contenu dans $\mathrm{J}$}

\begin{pro}\label{carac}
Soit $\phi$ un \'el\'ement de $\mathrm{J}_0$. Alors 
\begin{itemize}
\item ou bien $\mathrm{C}(\phi)$ est contenu dans $\mathrm{J};$

\item ou bien $\phi$ est p\'eriodique.
\end{itemize}
\end{pro}

\begin{proof}
Soit $\phi=(\psi(x,y),y)$ une transformation du groupe de Jonqui\`eres qui pr\'eserve la fibra\-tion~$y=$ cte fibre \`a fibre, $\psi\in\mathrm{PGL}_2(\mathbb{C}(y))$.

\noindent Soit $\varphi=(P(x,y),Q(x,y))$ une transformation rationnelle qui commute \`a $\phi$. Si $\varphi$ n'est pas dans $\mathrm{J}$, alors $Q=$ cte est
 une fibration invariante fibre \`a fibre par $\phi$ diff\'erente de $y=$cte. Ainsi $\phi$ poss\`ede deux fibrations distinctes invariantes fibre \`a  fibre 
et est donc p\'eriodique.
En effet les intersections des fibres $y=$ cte et $Q=$ cte g\'en\'eriques sont de cardinal fini, uniform\'ement born\'e; or ces intersections sont invariantes 
par $\phi$ d'o\`u le 
r\'esultat.
\end{proof}

\noindent Dans la suite de cette section nous d\'emontrons le Th\'eor\`eme \ref{ababelien} qui affirme que le centralisateur $\mathrm{C}(\phi)$ d'un \'el\'ement 
$\phi$ de~$\mathrm{J}_0$ \`a croissance lin\'eaire est une extension finie de $\mathrm{Ab}(\phi)$.

\subsection{Cas des transformations de $\mathrm{J}_a$} 

\noindent Soit $\phi=(x+a(y),y)$ un \'el\'ement non trivial de $\mathrm{J}_a$, c.-\`a.d. $a\not=0$; \`a conjugaison pr\`es par $(a(y)x,y)$ on peut supposer 
que $\phi=(x+1,y)$. Le groupe $\mathrm{C}(\phi)$ se d\'eduit alors de l'\'enonc\'e suivant 
(\cite{BlDe}).

\begin{pro}\label{centtypea}
Le centralisateur de $\phi=(x+1,y)$ est $$\big\{\left(x+b(y),\nu(y)\right)\,\big\vert\,b\in\mathbb{C}(y),\,\nu\in\mathrm{PGL}_2(\mathbb{C})\big\}\simeq
\mathrm{J}_a\rtimes\mathrm{PGL}_2(\mathbb{C}).$$
\end{pro}

\begin{proof}
Comme $\phi$ est non p\'eriodique, la Proposition~\ref{carac} assure que toute transformation $\psi$ qui commute \`a $\phi$ est de la forme $(\psi_1(x,y),
\nu(y))$ avec~$\nu$ dans $\mathrm{PGL}_2(\mathbb{C})$. En \'ecrivant $\phi \psi=\psi\phi$ on obtient: $\psi_1(x+1,y)=\psi_1(x,y)+1$. Ceci conduit \`a 
$\frac{\partial \psi_1}{\partial x}(x+1,y)=\frac{\partial \psi_1}{\partial x}(x,y);$ on en d\'eduit que $\frac{\partial \psi_1}{\partial x}$ est une 
fonction de~$y$, c.-\`a.d. $\psi_1(x,y)=A(y)x+B(y)$. En r\'e\'ecrivant $\psi_1(x+1,y)=\psi_1(x,y)+1$ on obtient $A=1$. Ainsi~$\psi$ est du type $(x+B(y),
\nu(y))$ avec~$B$ dans $\mathbb{C}(y)$ et~$\nu$ dans~$\mathrm{PGL}_2(\mathbb{C})$.
\end{proof}

\begin{rem}
Le centralisateur d'un \'el\'ement non trivial $(x+b(y),y)$ est donc birationnellement conjugu\'e \`a $\mathrm{J}_a\rtimes\mathrm{PGL}_2(\mathbb{C})$, la 
conjugaison ayant lieu dans le groupe de Jonqui\`eres.
\end{rem}

\subsection{Cas des transformations de $\mathrm{J}_m$} On s'int\'eresse maintenant aux \'el\'ements du groupe $\mathrm{J}_m$. Si $a\in\mathbb{C}(y)$ 
est non constant, on note $\mathrm{stab}(a)$ le sous-groupe fini de $\mathrm{PGL}_2(\mathbb{C})$ constitu\'e des \'el\'ements  laissant $a$ invariant: 
$$\mathrm{stab}(a)=\big\{\nu\in\mathrm{PGL}_2(\mathbb{C})\,\vert\,a(\nu(y))=a(y)\big\}.$$ 
La finitude de $\mathrm{stab}(a)$ provient du fait suivant: les fibres de $a$ qui sont finies, de cardinal inf\'erieur ou \'egal au degr\'e de $a$, 
sont invariantes par
les \'el\'ements de $\mathrm{stab}(a)$. 
On introduit le sous-groupe $$\mathrm{Stab}(a)=\big\{\nu\in\mathrm{PGL}_2(\mathbb{C})\,\vert\,a(\nu(y))=a(y)^{\pm 1}\big\}.$$ On remarque que $\mathrm{stab}(a)$ 
est distingu\'e dans $\mathrm{Stab}(a)$. 

\begin{eg}
Si $k$ est un entier et $a(y)=y^k$, alors $\mathrm{stab}(a)=\big\{\omega^ky\,\vert\,\omega^k=1\big\}$ et $\mathrm{Stab}(a)$ est le groupe $\big\langle\frac{1}{y},
\,\omega^ky\,\vert\,\omega^k=1\big\rangle$.
\end{eg}

\noindent On note $\overline{\mathrm{stab}(a)}$ le groupe lin\'eaire $$\overline{\mathrm{stab}(a)}=\big\{(x,\nu(y))\,\vert\,\nu\in\mathrm{stab}(a)
\big\}.$$ Enfin par d\'efinition le groupe $\overline{\mathrm{Stab}(a)}$ est engendr\'e par $\overline{\mathrm{stab}(a)}$ et les \'el\'ements du type 
$\left(\frac{1}{x},\nu(y)\right)$, $\nu$ appartenant \`a~$\mathrm{Stab}(a)\setminus\mathrm{stab}(a)$.

\begin{pro}\label{centtypeb}
Soit $\phi=(a(y)x,y)$ un \'el\'ement non p\'eriodique de $\mathrm{J}_m$. 

\noindent Si $\phi$ est \`a croissance born\'ee, c.-\`a.d. si $a$ est une constante, le centralisateur de $\phi$ est $$\big\{\big(b(y)x,\nu(y)\big)
\,\big\vert\, b\in\mathbb{C}(y)^*, \, \nu\in\mathrm{PGL}_2(\mathbb{C})\big\}.$$ 

\noindent Si $\phi$ est \`a croissance lin\'eaire, alors $\mathrm{C}(\phi)=\mathrm{J}_m\rtimes\overline{\mathrm{Stab}(a)}$.
\end{pro}

\begin{rems}
\begin{itemize}
\item En g\'en\'eral le groupe $\overline{\mathrm{Stab}(a)}$ se r\'eduit \`a l'identit\'e; plus pr\'ecis\'ement dans l'ensemble des fractions 
rationnelles de degr\'e donn\'e $d$ la propri\'et\'e $\overline{\mathrm{Stab}(a)}=\{\mathrm{id}\}$ est g\'en\'erique. Par suite, \og g\'en\'eriquement\fg\, pour~$\phi$ 
dans $\mathrm{J}_m$, le groupe $\mathrm{C}(\phi)$ se r\'eduit \`a $\mathrm{J}_m=\mathrm{Ab}(\phi)$.

\item Si $\phi=(a(y)x,y)$ avec $a$ non constant, alors $\mathrm{C}(\phi)$ est une extension finie de $\mathrm{J}_m=\mathrm{Ab}(\phi)$.

\item Si $\phi=(ax,y)$, avec $a$ dans $\mathbb{C}^*$, on a encore $\mathrm{C}(\phi)=\mathrm{J}_m\rtimes\overline{\mathrm{Stab}(a)}$ 
puisqu'ici on peut d\'efinir $\mathrm{Stab(a)}=\mathrm{PGL}_2(\mathbb{C})$.
\end{itemize}
\end{rems}

\begin{proof}
Dans un premier temps supposons que $\phi=(\alpha x,y)$ avec $\alpha$ dans $\mathbb{C}^*$ (non racine de l'unit\'e). Soit~$\psi$ une 
transformation birationnelle qui commute \`a $\phi;$ on peut, d'apr\`es la Proposition \ref{carac}, \'ecrire $\psi$ sous la forme 
$(P(x,y),Q(y))$. \'Etant donn\'e un point $(x_0,y_0)$ fix\'e on obtient en \'ecrivant la commutation de $\phi^n$ et $\psi$ l'\'egalit\'e 
$P(\alpha^nx_0,y_0)=\alpha^nP(x_0,y_0)$ d'o\`u $\frac{\partial P}{\partial x}(\alpha^nx_0,y_0)=\frac{\partial P}{\partial x}(x_0,y_0)$. 
Autrement dit $\frac{\partial P}{\partial x}$ est constante sur l'adh\'erence de Zariski de $\big\{(\alpha^nx_0,y_0)\,\vert\, n\in~\mathbb{Z}
\big\}$. Il en r\'esulte que $\frac{\partial P}{\partial x}$ est une fonction de $y;$ ainsi $P(x,y)$ est du type~$a(y)x+b(y)$. En 
r\'e\'ecrivant la commutation de $\phi^n$ et~$\psi$ on obtient $b(y)=\alpha^nb(y)$ ce qui conduit n\'ecessairement \`a $b=0$, c.-\`a.d. 
$\psi$ s'\'ecrit $(a(y)x,\nu(y))$.

\medskip

\noindent Supposons que $\phi$ s'\'ecrive $(a(y)x,y)$ avec $a$ dans $\mathbb{C}(y)\setminus\mathbb{C}^*$. Tout \'el\'ement $\psi$ de 
$\mathrm{C}(\phi)$ pr\'eserve la fibra\-tion~$y=$~cte (Proposition \ref{carac}):
\begin{align*}
&\psi=\left(\frac{A(y)x+B(y)}{C(y)x+D(y)},\nu(y)\right), && \left[\begin{array}{cc}A & B\\ C & D\end{array}\right]\in\mathrm{PGL}_2(
\mathbb{C}[y]), \,\nu\in\mathrm{PGL}_2(\mathbb{C}).
\end{align*}

\noindent En \'ecrivant explicitement la commutation de $\psi$ et $\phi$ on obtient:
\begin{align*}
&A(y)C(y)=A(y)C(y)a(\nu(y)) &&\text{et} && B(y)D(y)=B(y)D(y)a(\nu(y));
\end{align*}
on en d\'eduit que $AC$ et $BD$ sont nuls. Par suite $A=D=0$ ou $B=C=0$. Supposons dans un premier temps que~$B=C=0$ c'est-\`a-dire 
au changement de notation pr\`es $\psi=(A(y)x,\nu(y))$. La commutation de~$\phi$ et $\psi$ entra\^ine $a(\nu(y))=a(y)$. Comme 
$\overline{\mathrm{stab}(a)}\subset\mathrm{C}(\phi)$, il en r\'esulte que $\psi$ est dans le produit semi-direct~$\mathrm{J}_m
\rtimes~\overline{\mathrm{stab}(a)}$. Si maintenant $A=D=0$, la transformation $\psi$ est du type $\left(\frac{B(y)}{x},\nu(y)
\right)$ et cette fois la commutation entra\^ine~$a(\nu(y))=a(y)^{-1}$. Puisque $\overline{\mathrm{Stab}(a)}$ est contenu dans 
$\mathrm{C}(\phi)$, ici encore $\psi$ est dans $\mathrm{J}_m\rtimes\overline{\mathrm{Stab}(a)}$.
\end{proof}

\subsection{Cas des transformations $\phi$ de $\mathrm{J}_F$}

\noindent Dans ce qui suit on pr\'esente le calcul de $\mathrm{C}(\phi)$ pour les \'el\'ements de~$\mathrm{J}_F$ tels que~$F$ 
soit \`a racines simples (comme on l'a dit on s'y ram\`ene via conjugaison par un \'el\'ement de la forme $(a(y)x,y)$).
On \'ecrit donc $\phi$ sous la forme $\phi=\left(\frac{c(y)x+F(y)}{x+c(y)},y\right)$ o\`u $c\in\mathbb{C}(y)$; la courbe de 
points fixes $\mathcal{C}$ de $\phi$ est donn\'ee par $x^2=F(y)$. Puisque les valeurs propres de $\mathrm{M}_\phi=\left
[\begin{array}{cc}c(y) & F(y)\\ 1 & c(y)\end{array}\right]$ sont $c(y)\pm\sqrt{F(y)}$ on constate que $\phi$ est p\'eriodique 
si et seulement si $c$ est identiquement nul et dans ce cas $\phi$ est de p\'eriode $2$. Dans la suite on suppose $\phi$ non 
p\'eriodique. Comme $F$ est \`a racines simples le genre de $\mathcal{C}$ est sup\'erieur ou \'egal \`a $2$ pour $\deg F\geq 5$, 
vaut $1$ pour $\deg F\in\{3,\,4\}$; enfin $\mathcal{C}$ est rationnelle lorsque~$\deg F\in\{1,\,2\}$.

\subsubsection{Cas o\`u $\mathcal{C}$ est de genre strictement positif} 

\noindent Puisque $\phi$ est \`a croissance lin\'eaire, $\phi$ n'est pas p\'eriodique. Sur une fibre g\'en\'erale $\phi$ a deux 
points fixes qui correspondent aux deux points sur la courbe $x^2=F(y)$. Les cour\-bes~$x^2=F(y)$ et les fibres $y=$ cte sont 
invariantes par $\phi$ et il n'y a aucune autre courbe invariante m\^eme de genre $0$ ou $1$. En effet une courbe invariante 
diff\'erente de $y=$ cte intersecte une fibre g\'en\'erale en un nombre fini de points forc\'ement invariants par $\phi$; comme 
$\phi$ est d'ordre infini c'est impossible car une transformation de Moebius qui laisse invariant un ensemble \`a plus de trois 
\'el\'ements est p\'eriodique. 

\begin{pro}\label{genregd}
Soit $\phi$ l'\'el\'ement de $\mathrm{J}_F$ donn\'e par $\left(\frac{c(y)x+F(y)}{x+c(y)},y\right)$. Supposons que
$\phi$ soit non p\'eriodique $($c.-\`a.d. $c\not=0)$ et que $F$ soit un polyn\^ome 
\`a racines simples de degr\'e sup\'erieur ou \'egal \`a $3$ $($c.-\`a.d. le genre de $\mathcal{C}$ est sup\'erieur 
ou \'egal \`a $1)$. Alors $\mathrm{C}(\phi)$ est une extension finie de $\mathrm{J}_F=\mathrm{Ab}(\phi)$, qui est
triviale pour la plupart des $F$.
\end{pro}

\begin{proof}
Sur chaque fibre g\'en\'erique ($y=y_0$ avec $F(y_0)\not=0$) la restriction $\phi_{\vert y=y_0}$ de $\phi$ a deux 
points fi\-xes~$(\pm\sqrt{F(y_0)},y_0)$. La Proposition \ref{carac} assure que $\mathrm{C}(\phi)\subset
\mathrm{J}$. Nous allons nous int\'eresser au noyau de $$\pi_{2\vert_{\mathrm{C}(\phi)}}\colon\mathrm{C}(\phi)\to
\mathrm{PGL}_2(\mathbb{C}),$$ c.-\`a.d. aux \'el\'ements $\psi$ de $\mathrm{C}(\phi)$ qui pr\'eservent la fibration
 $y=$ cte fibre \`a fibre. Notons que $\psi$ pr\'eserve $\mathcal{C}$ et que l'automor\-phisme~$\psi_{\vert\mathcal{C}}$
 de $\mathcal{C}$ respecte les points fixes pr\'ec\'edents. Par cons\'equent ou bien $\psi_{\vert\mathcal{C}}$ est 
l'identit\'e de $\mathcal{C}$, ce qui revient \`a dire que $\psi$ est dans $\mathrm{J}_F$, ou bien $\psi_{\vert
\mathcal{C}}$ est l'involution $(-x,y)$ de $\mathcal{C}$. On remarque que~$(-x,y)$ est r\'ealis\'ee par l'involution 
globale $\tau=\left(-\frac{F(y)}{x},y\right);$ ainsi toute transformation birationnelle pr\'eservant~$\mathcal{C}$
 et la fibration $y=$ cte fibre \`a fibre est ou bien dans $\mathrm{J}_F$, ou bien dans $\tau\,\mathrm{J}_F$.

\noindent Un calcul \'el\'ementaire montre que $\tau\phi\tau^{-1}=\tau\phi\tau=\phi^{-1}$ si bien que $\tau$ n'est pas 
dans $\mathrm{C}(\phi)$. Il en r\'esulte que~$\ker\pi_{2\vert_{\mathrm{C}(\phi)}}=\mathrm{J}_F$.

\noindent Consid\'erons un \'el\'ement $\varphi$ de $\mathrm{C}(\phi)$. Comme $\varphi$ doit pr\'eserver $\mathcal{C}$ 
et la fibration $y=$ cte la restriction $\varphi_{\vert\mathcal{C}}$ de~$\varphi$ \`a $\mathcal{C}$ est un automorphisme 
de $\mathcal{C}$ qui commute \`a l'involution $\tau_{\vert\mathcal{C}}$. En g\'en\'eral sur $F$ le groupe 
$\mathrm{Aut}_\tau(\mathcal{C})$ de tels automorphismes se r\'eduit \`a l'identit\'e et $\tau_{\vert\mathcal{C}}$. Dans 
tous les cas $\mathrm{Aut}_\tau(\mathcal{C})$ est un groupe fini. C'est bien s\^ur \'evident lorsque le genre de 
$\mathcal{C}$ est sup\'erieur ou \'egal \`a $2$ puisque le groupe des automorphismes de $\mathcal{C}$ est fini; dans 
le cas elliptique, on conclut en utilisant le fait que les $\psi_{\vert\mathcal{C}}$ laissent invariants les points 
de ramification.
\end{proof}

\begin{rems}
\begin{itemize}
\item Supposons que $F(y)=\displaystyle\prod_{i=1}^s(y^k-y_i)$ o\`u $k$ d\'esigne un entier positif et $y_i$ des complexes distincts 
non nuls. Si $\phi=\left(\frac{c(y)x+F(y)}{x+c(y)},y\right)$, avec $c\not=0$, alors $\mathrm{C}(\phi)$ contient $\mathrm{G}=\big\{(x,
\omega y)\,\vert\,\omega^k=1\big\}$ d\`es que~$\mathrm{G}$ est contenu dans $\overline{\mathrm{stab}(c)}$ autrement dit d\`es que $c$ 
est invariant par les $\omega y$, $\omega^k=1$.
  
\item Plus g\'en\'eralement consid\'erons une transformation $\phi=\left(\frac{c(y)x+F(y)}{x+c(y)},y\right)$. Supposons que
 $\mathrm{stab}(c)$ et $\mathrm{stab}(F)$ contiennent le groupe $\langle \frac{1}{y},\,\omega y\,\vert\,\omega^k=1\rangle$; c'est 
par exemple le cas   lorsque $c$ et $F$ se factorisent dans $y^k+\frac{1}{y^k}$. Alors~$\mathrm{C}(\phi)$ contient le groupe 
di\'edral $\langle \left(x,\frac{1}{y}\right),\,(x,\omega y)\,\vert\,\omega^k=1 \rangle$. Notons qu'ici, par commodit\'e, nous avons 
choisi une \'ecriture non polynomiale de $F$.
\end{itemize}
\end{rems}

\subsubsection{Cas o\`u $\mathcal{C}$ est rationnelle}\label{Subsec:crationnelle} 

\noindent Soit $\phi$ un \'el\'ement de $\mathrm{J}_F$ \`a croissance lin\'eaire.
\noindent Comme on l'a vu la courbe de points fixes $\mathcal{C}$ de $\phi$ est donn\'ee par $x^2=F(y)$. Soit $\psi$ un \'el\'ement de 
$\mathrm{C}(\phi)$. Alors ou bien~$\psi$ contracte $\mathcal{C}$, ou bien~$\psi$ pr\'eserve~$\mathcal{C}$. Or la Proposition \ref{carac} 
assure que $\psi$ pr\'eserve la fibration $y=$ cte; la courbe $\mathcal{C}$ \'etant transverse \`a la fibration, $\psi$ ne peut pas 
contracter $\mathcal{C}$ donc $\psi$ est un \'el\'ement du groupe de Jonqui\`eres qui pr\'eserve~$\mathcal{C}$. D\`es que $F$ est de 
degr\'e sup\'erieur ou \'egal \`a $3$ les hypoth\`eses de la Proposition \ref{genregd} sont satisfaites; supposons donc que~$\deg F\leq 2$. 
Le cas $\deg F=2$ se ram\`ene \`a $\deg F=1$. En effet, soit $\phi$ un \'el\'ement du type $\left(\frac{c(y)x+y}{x+c(y)},y\right)$; posons 
$\varphi=\left(\frac{x}{cy+d}, \frac{ay+b}{cy+d}\right)$. On peut v\'erifier que $\varphi^{-1}\phi \varphi$ est du type $\left(
\frac{\widetilde{c}(y)x+(ay+b)(cy+d)}{x+\widetilde{c}(y)},y\right)$, ce qui permet d'atteindre tous les polyn\^omes quadratiques $F$ \`a 
racines simples. Si $\deg F=1$, c.-\`a.d. de la forme $ay+b$, on se ram\`ene, quitte \`a conjuguer $\phi$ par $\left(x,\frac{y-b}{a}\right)$, \`a $F(y)=y$.

\begin{lem}\label{secondcomp}
Soit $\phi$ une transformation du type $\left(\frac{c(y)x+y}{x+c(y)},y\right)$ avec $c$ dans $\mathbb{C}(y)^*$. Si $\psi$ est un \'el\'ement 
de~$\mathrm{C}(\phi)$, alors~$\pi_2(\psi)$ est ou bien du type $\frac{\alpha}{y}$, $\alpha\in\mathbb{C}^*$, ou bien du type $\xi y$, $\xi$ 
racine de l'unit\'e; de plus, $\pi_2(\psi)$ est contenu dans le groupe fini $\mathrm{stab}\left(\frac{4c(y)^2}{c(y)^2-y}\right)$.
\end{lem}

\begin{proof} 
\'Ecrivons $\psi$ sous la forme $\left(\frac{A(y)x+B(y)}{C(y)x+D(y)},\pi_2(\psi)\right)$. 

\noindent La conique $\mathcal{C}$ d\'ecrite par $y=x^2$ est munie de l'involution $\iota\colon(x,y)\mapsto(-x,y)$ qui a ses points 
fixes en $(0,0)$ et $(0,\infty)\in\mathcal{C}$. On remarque que la restriction $\psi_{\vert\mathcal{C}}\colon\mathcal{C}\to\mathcal{C}$ 
est un automorphisme de $\mathcal{C}$ qui commute \`a $\iota_{\vert\mathcal{C}}$. En particulier les points fixes de $\iota$ sont 
pr\'eserv\'es par $\psi_{\vert\mathcal{C}}$ dans leur ensemble. Il en r\'esulte que $\pi_2(\psi)$ laisse invariant l'ensemble $\{0,\,\infty\}$,
 c.-\`a.d. ou bien $\pi_2(\psi)=\alpha y$, ou bien $\pi_2(\psi)=\frac{\alpha}{y}$.

\noindent La commutation de $\phi$ et $\psi$ se traduit par l'\'egalit\'e suivante dans $\mathrm{PGL}_2(\mathbb{C}(y))$:
$$\left[\begin{array}{cc}c(\pi_2(\psi)) & \pi_2(\psi)\\ 1 & c(\pi_2(\psi))\end{array}\right]=\left[\begin{array}{cc}A(y) & B(y)\\ C(y) & D(y)
\end{array}\right]\mathrm{M}_\phi\left[\begin{array}{cc}A(y) & B(y)\\ C(y) & D(y)\end{array}\right]^{-1}.$$ Ceci implique que les matrices 
$\mathrm{M}_\phi=\left[\begin{array}{cc}c(y) & y \\ 1 & c(y)\end{array}\right]$ et $\mathrm{M}_\phi(\pi_2(\psi))=\left[\begin{array}{cc}
c(\pi_2(\psi)) & \pi_2(\psi) \\ 1 & c(\pi_2(\psi))\end{array}\right]$ sont conjugu\'ees dans~$\mathrm{PGL}_2(\mathbb{C}(y))$. En particulier 
on a l'\'egalit\'e $$\mathrm{BB}(\phi)\big(\pi_2(\psi)\big)=\mathrm{BB}(\phi)(y)=\frac{4c(y)^2}{c(y)^2-y}.$$ 
\end{proof}

\noindent Pour $\alpha$ non nul on note $\mathrm{D}_\infty(\alpha)$ le groupe di\'edral infini $$\mathrm{D}_\infty(\alpha)=\Big\langle\frac{\alpha}{y},
\,\omega y\,\big\vert\,\omega\text{ racine de l'unit\'e}\Big\rangle;$$ remarquons que les $\mathrm{D}_\infty(\alpha)$ sont conjugu\'es \`a $\mathrm{D}_\infty(1)$. 

\noindent Si $c$ est un \'el\'ement non constant de $\mathbb{C}(y)^*$ on d\'esigne par $\mathrm{S}(c;\alpha)$ le sous-groupe fini de 
$\mathrm{PGL}_2(\mathbb{C})$ donn\'e par $$\mathrm{S}(c;\alpha)=\mathrm{stab}\left(\frac{4c(y)^2}{c(y)^2-y}\right)\cap\mathrm{D}_\infty(\alpha).$$

\noindent La description des $\mathrm{C}(\phi)$ avec $\phi$ dans $\mathrm{J}_F$ et $\mathcal{C}=\mathrm{Fix}\,\phi$ rationnelle se ram\`ene \`a 
l'\'enonc\'e suivant.

\begin{pro}\label{genrenul}
Soit $\phi=\left(\frac{c(y)x+y}{x+c(y)},y\right)$ avec $c$ dans $\mathbb{C}(y)^*$, $c$ non constant. Il existe $\alpha$ dans $\mathbb{C}^*$ tel 
que $$\mathrm{C}(\phi)=\mathrm{J}_y\rtimes\mathrm{S}(c;\alpha).$$
\noindent 
\end{pro}

\begin{proof}
Supposons que $\mathrm{S}(c;\alpha)$ soit r\'eduit \`a l'identit\'e pour tout $\alpha$, autrement dit $\pi_2(\mathrm{C}(\phi))=\{\mathrm{id}\}$
 d'apr\`es le Lemme \ref{secondcomp}. En \'ecrivant que les \'el\'ements $\psi$ de $\mathrm{C}(\phi)$ pr\'eservent la courbe $\mathcal{C}$ on 
obtient $A=D$ et~$B(y)=C(y)y$, ce qui signifie que $\mathrm{C}(\phi)=\mathrm{J}_y$ et prouve l'\'enonc\'e dans ce cas.

\noindent Si maintenant $\pi_2(\psi)$ est non trivial pour un certain $\psi$ dans $\mathrm{C}(\phi)$, on sait que $\pi_2(\psi)$ est ou bien du type 
$\frac{\alpha}{y}$, ou bien du type~$\omega y$ avec $\omega$ racine de l'unit\'e (Lemme \ref{secondcomp}). Supposons que $\pi_2(\psi)=\omega y$ 
avec  $\omega$ racine de l'unit\'e. De l'invariance de~$\frac{4c^2}{c^2-y}$ par $\omega y$ on tire $$c(\omega y)^2=\omega c(y)^2.$$ Par suite il 
existe une racine carr\'ee $\upsilon$ de $\omega$, $\upsilon^2=\omega$, telle que $c(\upsilon^2 y)=\upsilon c(y).$ On constate alors que la 
transformation lin\'eaire $\varphi=(\upsilon x,\upsilon^2y)$ est dans $\mathrm{C}(\phi)$. On peut appliquer \`a $\psi\varphi^{-1}$, qui v\'erifie 
$\pi_2(\psi\varphi^{-1})=\mathrm{id}$, le premier argument, de sorte que $\psi\varphi^{-1}$ est dans $\mathrm{J}_y$. Enfin si $\pi_2(\psi)=
\frac{\alpha}{y}$ pour un certain $\alpha$ on proc\`ede comme ci-dessus. On a $c\left(\frac{\alpha}{y}\right)^2=\frac{\alpha}{y^2}c(y)^2$ et il 
existe une racine carr\'ee $\beta$ de $\alpha$, $\beta^2=\alpha$, telle que $c\left(\frac{\beta^2}{y}\right)=~\frac{\beta}{y}c(y)$. On remarque 
cette fois que la transforma\-tion~$\varphi=\left(\frac{\beta x}{y},\frac{\beta^2}{y}\right)$ commute \`a $\phi$; de plus $\psi\varphi^{-1}$ est dans 
$\mathrm{C}(\phi)$ et v\'erifie~$\pi_2(\psi\varphi^{-1})=\mathrm{id}$. Par suite $\psi\varphi^{-1}$ est dans $\mathrm{J}_y$.
\end{proof}

\begin{rem}
Si l'inclusion $\mathrm{J}_y\subset\mathrm{C}(\phi)$ est stricte on peut \'evidemment affiner les formes normales de~$\left(\frac{c(y)x+y}{x+c(y)},y
\right)$ suivant la nature de $\mathrm{S}(c;\alpha)$.
\end{rem}

\section{Centralisateurs des \'el\'ements de $\mathrm{J}$ et \'equations aux diff\'erences}\label{Sec:centralisateursj}

\noindent Soit $\phi$ un \'el\'ement de $\mathrm{J}\setminus\mathrm{J}_0$ qui pr\'eserve une unique fibration rationnelle. Consid\'erons le morphisme 
$$\pi_{2\vert_{\mathrm{C}(\phi)}}\colon\mathrm{C}(\phi)\to\mathrm{PGL}_2(\mathbb{C})$$ qui \`a une transformation $\psi$ de $\mathrm{C}(\phi)$ associe 
la seconde composante $\pi_2(\psi)$ de $\psi$; soit $\mathrm{H}$ le noyau de $\pi_2$. Remarquons que si $\mathrm{H}$ est trivial alors $\mathrm{C}(\phi)$ 
est isomorphe \`a $\pi_2(\mathrm{C}(\phi))$. On peut se ramener par conjugaison \`a $\pi_2(\phi)=y+1$, ou~$\pi_2(\phi)=\alpha y$. Notons que $\pi_2(\mathrm{C}(\phi))$ 
est contenu dans le centralisateur (dans $\mathrm{PGL}_2(\mathbb{C})$) de $\pi_2(\phi)$; ce groupe est ab\'elien, isomorphe \`a $\mathbb{C}$ ou 
$\mathbb{C}^*$, sauf dans le cas $\pi_2(\phi)=-y$ o\`u il est isomorphe \`a $\mathbb{C}^*\rtimes\mathbb{Z}/2\mathbb{Z}$. Nous faisons une 
description au cas par cas dans \S\ref{pasdetorsion} et \S\ref{torsion}.

\subsection{Cas o\`u $\mathrm{H}$ n'est pas de torsion}\label{pasdetorsion}

\subsubsection{\'Etude du cas $\pi_2(\phi)=y+1$}\label{pasdetorsion:transl}

\noindent Soit $\psi$ un \'el\'ement d'ordre infini dans $\mathrm{H}$. Puisque $\phi$ appartient \`a $\mathrm{C}(\psi)$ on peut utiliser la classification 
\'etablie au \S\ref{calculcentralisateur}. Le fait que $\pi_2(\phi)$ soit une translation implique que $\mathrm{im}\,\pi_2$ est non fini et donc $\psi$ est
 n\'ecessairement, \`a conjugaison pr\`es, de l'un des deux types suivants
\begin{align*}
&\psi=(x+1,y), &&\psi=(\alpha x,y)
\end{align*}
$\alpha$ n'\'etant pas une racine de l'unit\'e.
Dans ces deux \'eventualit\'es les calculs de $\mathrm{C}(\psi)$ nous indique que $\phi$ est du type suivant (respectivement):
\begin{align*}
& (x+a(y),y+1), && (b(y)x,y+1).
\end{align*}
Nous allons \'etudier le centralisateur de $\phi$ dans ces deux cas. Rappelons que $\mathrm{C}(\phi)$ est contenu dans $\mathrm{J}$ puisque $\phi$
pr\'eserve une seule fibration rationnelle.

\medskip

\noindent Dans un premier temps consid\'erons la possibilit\'e $\phi=(x+a(y),y+1)$. 

\noindent Commen\c{c}ons par remarquer que si l'\'equation aux diff\'erences $f(y)-f(y+1)=a(y)$ poss\`ede une solution rationnelle, alors $\phi$ 
est conjugu\'ee \`a $(x,y+1)$ qui pr\'eserve plus d'une fibration. 

\noindent Notons que si $\mathrm{C}(\phi)$ contient un \'el\'ement du type $(x+b(y),y)$ alors $b(y+1)=b(y)$, \emph{i.e.} $b$ est constant; 
d'ailleurs tous les $(x+\alpha,y)$ avec $\alpha$ dans $\mathbb{C}$ commutent \`a $\phi$. D\'ecrivons plus g\'en\'eralement les \'el\'ements 
$\varphi=\left(\frac{A(y)x+B(y)}{C(y)x+D(y)},y\right)$ de $\ker\pi_2$. Si~$C=0$, la commutation de $\varphi$ et $\phi$ conduit \`a 
\begin{align*}
& A(y)=A(y+1)\,\,\,\,\,(\star)&& \text{et}&& (A(y)-1)a(y)=B(y)-B(y+1)\,\,\,\,\,(\diamond).
\end{align*}
L'\'egalit\'e $(\star)$ implique que $A$ est constant: $A(y)=\alpha\in\mathbb{C}^*$. Alors $(\diamond)$ se r\'e\'ecrit $(\alpha-1)a(y)=B(y)-
B(y+~1)$. N\'ecessairement $\alpha=1$ (comme on l'a vu l'\'equation aux diff\'erences $f(y)-f(y+1)=a(y)$ n'a pas de solution rationnelle sous 
l'hypoth\`ese $\phi$ pr\'eserve une unique fibration); par suite $B$ est constante, \emph{i.e.} $\varphi=(x+\beta,y)$ avec~$\beta$ dans $\mathbb{C}$. 
Supposons maintenant que $C$ soit non nul; on peut alors se ramener \`a $C=1$. On peut v\'erifier que $\phi\varphi=\varphi\phi$ implique $A(y+1)-
A(y)=a(y)$; il n'y a donc pas d'\'el\'ement de la forme $\left(\frac{A(y)x+B(y)}{x+D(y)},y\right)$ dans~$\ker\pi_{2\vert_{\mathrm{C}(\phi)}}$. Ainsi 
$\ker\pi_{2\vert_{\mathrm{C}(\phi)}}=\big\{(x+\beta,y)\,\vert\,\beta\in\mathbb{C}\big\}$ et on obtient une description de~$\mathrm{C}(\phi)$ via la 
suite exacte $$0\longrightarrow \ker\pi_{2\vert_{\mathrm{C}(\phi)}}\simeq\mathbb{C}\longrightarrow\mathrm{C}(\phi)\longrightarrow\mathrm{im}\,
\pi_{2\vert_{\mathrm{C}(\phi)}}\longrightarrow 0.$$ Comme $\mathrm{im}\,\pi_2\subset\mathbb{C}$, on constate que $[\mathrm{C}(\phi),\mathrm{C}(\phi)]$ est ab\'elien et donc 
$\mathrm{C}(\phi)$ r\'esoluble.

\begin{pro}
Si $\phi=(x+a(y),y+1)$ pr\'eserve une seule fibration, alors $\mathrm{C}(\phi)$ est r\'esoluble m\'etab\'elien. 
\end{pro}

\noindent Consid\'erons maintenant l'\'eventualit\'e $\phi=(b(y)x,y+1)$. Remarquons que l'\'equation aux diff\'erences $\frac{f(y+1)}{f(y)}=~b(y)$ n'a 
pas de solution sinon $\phi$ serait conjugu\'e \`a $(x,y+1)$ et poss\`ederait plus d'une fibration invariante. Soit $\varphi$ un \'el\'ement de 
$\ker\pi_2$, il est du type $\left(\frac{A(y)x+B(y)}{C(y)x+D(y)},y\right)$. Si $C=0$, on peut supposer que~$D=~1$; en \'ecrivant que $\varphi$ et 
$\phi$ commutent on obtient 
\begin{align*}
&  A(y+1)=A(y)\,\,\,\,\, (\star) &&\text{et}&& B(y+1)=B(y)b(y)\,\,\,\,\, (\star\star)
\end{align*}
L'\'egalit\'e $(\star)$ implique que $A$ est une constante. Puisque l'\'equation aux diff\'erences $\frac{f(y+1)}{f(y)}=b(y)$ n'a pas de solution, 
$(\star\star)$ entra\^ine que $B=0$. Autrement dit $\varphi$ s'\'ecrit $(\alpha x,y)$. Si $C$ est non nul, on peut se ramener \`a~$C=1$. La 
commutation de $\varphi$ et $\phi$ entra\^ine les \'egalit\'es
\begin{align*}
&  A(y+1)=A(y)b(y) &&\text{et}&& B(y+1)D(y)=b(y)B(y)D(y+1)
\end{align*}
 d'o\`u $A=D=0$. La derni\`ere condition impos\'ee par $\varphi\phi=\phi\varphi$ est $$B(y+1)=b(y)^2B(y)\,\,\,\,\,(\diamond)$$ Il se peut 
que $(\diamond)$ ait une solution par exemple pour $b(y)=-\frac{y+1}{y}$; dans ce cas pr\'ecis $(\deg\phi^n)_n$ n'est pas \`a croissance 
born\'ee. On constate que deux solutions de $(\diamond)$ diff\`erent d'une constante multiplicative. Ainsi g\'en\'eriquement $\ker\pi_{2
\vert_{\mathrm{C}(\phi)}}=\big\{(\beta x,y)\,\vert\, \beta\in\mathbb{C}^*\big\}$ et dans le cas o\`u $(\diamond)$ a une solution $$\ker
\pi_{2\vert_{\mathrm{C}(\phi)}}=\left\langle(\beta x,y),\left(\frac{B(y)}{x},y\right)\,\big\vert\,\beta \in\mathbb{C}^*\right\rangle.$$ On 
obtient encore une description de $\mathrm{C}(\phi)$ via la suite exacte $$0\longrightarrow \ker\pi_{2\vert_{\mathrm{C}(\phi)}}
\longrightarrow\mathrm{C}(\phi)\longrightarrow\mathrm{im}\,\pi_{2\vert_{\mathrm{C}(\phi)}}\longrightarrow 0.$$
 
\subsubsection{\'Etude du cas $\pi_2(\phi)=\beta y$, $\beta$ d'ordre infini}

\noindent Soit $\psi$ un \'el\'ement d'ordre infini dans $\mathrm{H}$. Puisque $\phi$ appartient \`a $\mathrm{C}(\psi)$ on peut utiliser la 
classification \'etablie au \S\ref{calculcentralisateur}. Comme pr\'ec\'edemment on se ram\`ene aux deux \'eventualit\'es:
\begin{align*}
& (x+a(y),\beta y), && (b(y)x,\beta y).
\end{align*}

\medskip

\noindent Dans le cas o\`u $\phi=(x+a(y),\beta y)$, on obtient en utilisant le m\^eme raisonnement que  $$\ker\pi_{2\vert_{\mathrm{C}(\phi)}}=
\big\{(x+\alpha,y)\,\vert\,\alpha\in\mathbb{C}\big\}$$ et on a une description de~$\mathrm{C}(\phi)$ via $$0\longrightarrow \ker\pi_{2\vert_{
\mathrm{C}(\phi)}}\simeq\mathbb{C}\longrightarrow\mathrm{C}(\phi)\longrightarrow\mathrm{im}\,\pi_{2\vert_{\mathrm{C}(\phi)}}\longrightarrow 0.$$ 

\medskip

\noindent Lorsque $\phi=(b(y)x,\beta y)$ alors $\ker\pi_{2\vert\mathrm{C}(\phi)}=\big\{(\alpha x,y)\,\vert\, \alpha\in\mathbb{C}^*\big\}$  
d'o\`u une description de $\mathrm{C}(\phi)$ via la suite exacte $$0\longrightarrow \ker\pi_{2\vert_{\mathrm{C}(\phi)}}\simeq\mathbb{C}^*
\longrightarrow\mathrm{C}(\phi)\longrightarrow\mathrm{im}\,\pi_{2\vert_{\mathrm{C}(\phi)}}\longrightarrow 0.$$ En effet, pour des raisons 
analogues \`a celles \'evoqu\'ees pr\'ec\'edemment, l'\'equation aux diff\'erences $\frac{f(\beta y)}{f(y)}=~b(y)$ n'a pas de solution. 
Soit $\varphi$ un \'el\'ement de $\ker\pi_2$, il est du type $\left(\frac{A(y)x+B(y)}{C(y)x+D(y)},y\right)$. Si $C=0$, alors $\varphi$ 
s'\'ecrit $(\alpha x,y)$. Si $C$ est non nul, on peut se ramener \`a $C=1$. La commutation de $\varphi$ et $\phi$ entra\^ine $A=D=0$. La 
derni\`ere condition impos\'ee par~$\varphi\phi=\phi\varphi$ est $B(\beta y)=b(y)^2B(y)$ mais celle-ci n'est pas compatible avec le fait 
que $\phi$ soit \`a croissance lin\'eaire. En effet, $\phi$ est \`a croissance lin\'eaire si et seulement si $\psi=(b(y)^2x,\beta y)$ l'est 
puisque $$\psi^n(x,y)=\left(\prod_{j=0}^{n-1}b(\beta^jy)^2x,\beta^ny\right)=\left(\left(\prod_{j=0}^{n-1}b(\beta^jy)\right)^2x,\beta^ny\right);$$ 
mais $\displaystyle\prod_{j=0}^{n-1}b(\beta^jy)^2=\frac{B(\beta^{n-1}y)}{B(y)}$ est \`a croissance born\'ee.

\begin{pro}
Si $\phi=(x+a(y),\beta y)$ $($resp. $(b(y)x,\beta y)$$)$ avec $\beta$ d'ordre infini pr\'eserve une seule fibration, alors $\mathrm{C}(\phi)$ est r\'esoluble m\'etab\'elien. 
\end{pro}

\subsection{Cas o\`u $\mathrm{H}=\ker\pi_{2_{\vert_{\mathrm{C}(\phi)}}}$ de torsion}\label{torsion}  

\noindent Donnons une description des sous-groupes de torsion infinis de~$\mathrm{PGL}_2(\mathbb{C}(y))$, elle d\'ecoule de \cite[Th\'eor\`eme 2]{Bl2}.

\begin{lem}
Les sous-groupes de torsion infinis de $\mathrm{J}_0\simeq\mathrm{PGL}_2(\mathbb{C}(y))$ sont \`a conjugaison pr\`es de la forme 
\begin{equation}\label{gpetorsioninf}
\mathrm{G}_a=\langle\left(\frac{a(y)}{x},y\right),\,(\omega x,y)\,\vert\,\omega\in\Lambda\rangle
\end{equation}
o\`u $a$ d\'esigne un \'el\'ement de $\mathbb{C}(y)$ $($\'eventuellement nul$)$ et $\Lambda$ un sous-groupe infini de racines de l'unit\'e.
\end{lem}

\begin{pro}
Soit $\phi$ une transformation de $\mathrm{J}\setminus\mathrm{J}_0$ qui pr\'eserve une unique fibration rationnelle. Supposons que le noyau de $\pi_{2\vert_{\mathrm{C}(\phi)}}
\colon\mathrm{C}(\phi)\to\mathrm{PGL}_2(\mathbb{C})$ soit de torsion. Alors  $\ker\pi_2$ est fini et, \`a indice fini pr\`es, $\mathrm{C}(\phi)$ est isomorphe \`a~$\pi_2(
\mathrm{C}(\phi))$ qui est un sous-groupe ab\'elien de $\mathrm{PGL}_2(\mathbb{C})$. 
\end{pro}

\begin{proof}
Raisonnons par l'absurde: supposons que $\mathrm{H}=\ker\pi_2$ soit infini de torsion c'est-\`a-dire du type~$(\ref{gpetorsioninf})$. En passant \`a l'adh\'erence de Zariski 
on constate que $\phi$ commute aussi aux \'el\'ements du type $(\alpha x,y)$ o\`u $\alpha$ d\'esigne un \'el\'ement quelconque de $\mathbb{C}^*$: contradiction avec 
l'hypoth\`ese selon laquelle $\mathrm{H}$ de torsion.
\end{proof}

\subsection{Exemples}

\noindent Consid\'erons la transformation $\theta_1$ donn\'ee par $\theta_1=\left(x+\frac{1}{y},y+1\right)$. On remarque que $$\theta_1^n=\left(x
+\frac{1}{y}+\frac{1}{y+1}+\ldots+\frac{1}{y+n-1},y+n\right)$$ de sorte que la suite $(\deg\theta_1^n)_n$ n'est pas born\'ee. Il en r\'esulte que 
l'\'equation aux diff\'erences $f(y+1)-f(y)=\frac{1}{y}$ n'a pas de solution rationnelle. Comme on l'a vu l'image de $\pi_2$ est dans le groupe des 
translations $y+\tau$ et son noyau se r\'eduit \`a $\big\{(x+\alpha,y)\,\vert\,\alpha\in\mathbb{C}\big\}$ (\emph{voir} \S\ref{pasdetorsion:transl}). 

\noindent Consid\'erons maintenant un \'el\'ement g\'en\'eral $\psi=\left(\frac{A(y)x+B(y)}{C(y)x+D(y)},y+\tau\right)$ de $\mathrm{C}(\theta_1)$; 
on est ramen\'e \`a \'etudier les deux possibilit\'es suivantes
\begin{align*}
& (A(y)x+B(y),y+\tau)\,\, \,\,\,(\star), && \left(\frac{A(y)x+B(y)}{x+D(y)},y+\tau\right)\,\,\,\,\,(\diamond).
\end{align*}

\noindent On commence par l'\'eventualit\'e $(\star)$; la commutation s'\'ecrit alors $$A(y+1)\left(x+\frac{1}{y}\right)+B(y+1)=A(y)x+B(y)+\frac{1}
{y+\tau}$$ d'o\`u $A(y+1)=A(y)$ qui implique que $A$ est une constante $a\in\mathbb{C}$ et $$B(y+1)-B(y)=\frac{1}{y+\tau}-\frac{a}{y}\,\,\,\,\,(\star
\star).$$ Notons que la somme des r\'esidus de $B(y+1)-B(y)$ est nulle de sorte que $a$ vaut $1$. Supposons que $\tau$ ne soit pas entier. Alors la 
transformation $F=\left(x+\frac{1}{y+\tau}-\frac{1}{y},y+1\right)$ v\'erifie $$F^n=\left(x+\frac{1}{y+\tau}+\frac{1}{y+\tau+1}+\ldots+\frac{1}{y+n-1
+\tau}-\left(\frac{1}{y}+\ldots+\frac{1}{y+n-1}\right),y+n\right);$$ ainsi la suite $(\deg F^n)_n$ est non born\'ee. Ceci implique qu'il n'y a pas 
de solution \`a l'\'equation aux diff\'erences~$(\star\star)$. En effet si $B$ est solution de $(\star\star)$ la transformation $(x+B(y),y)$ conjugue 
$F$ \`a $(x,y+1)$ dont la suite des degr\'es des it\'er\'es est \`a croissance born\'ee. Il en r\'esulte que $\tau$ appartient \`a $\mathbb{Z}$. En 
composant $\psi$ par un it\'er\'e convenable de $\theta_1$ on se ram\`ene \`a un \'el\'ement de $\ker\pi_2$ ce qui termine le cas $(\star)$.

\noindent Pour $(\diamond)$ la commutation se traduit par $$\frac{A(y+1)\left(x+\frac{1}{y}\right)+B(y+1)}{x+\frac{1}{y}+D(y+a)}=\frac{A(y)x+B(y)}{x+
D(y)}+\frac{1}{y+\tau} $$ soit encore par $$\left(A(y+1)\left(x+\frac{1}{y}\right)+B(y+1)\right)\big(x+D(y)\big)\big(y+\tau\big)=\Big((A(y)x+B(y))(y+
\tau)+x+D(y)\Big)\left(x+\frac{1}{y}+D(y+1)\right).$$ En particulier $A(y+1)-A(y)=\frac{1}{y+\tau}$; comme pr\'ec\'edemment cette \'equation aux 
diff\'erences n'a pas de solution: un argument de croissance des degr\'es assure que $\left(x+\frac{1}{y+\tau},\frac{1}{y+1}\right)$ n'est pas conjugu\'e 
\`a $\left(x,\frac{1}{y+1}\right)$.

\noindent Ceci montre que $\mathrm{C}(\theta_1)$ est engendr\'e par les $(x+\alpha,y)$ et les it\'er\'es de $\theta_1$.

\noindent L'exemple que l'on vient d'\'etudier fait partie d'une famille de transformations birationnelles pour lesquelles on peut d\'ecrire les centralisateurs.

\begin{pro}
Consid\'erons la famille $(\theta_\alpha)_{\alpha\in\mathbb{C}^*}$ de transformations birationnelles donn\'ee par $\theta_\alpha=\left(\alpha x+\frac{1}{y},y+1\right)$.

\noindent Le centralisateur $\mathrm{C}(\theta_1)$ de $\theta_1$ est le groupe engendr\'e par $(x+a,y)$, $a$ appartenant \`a $\mathbb{C}$, et les it\'er\'es de $\theta_1$; 
c'est un groupe ab\'elien isomorphe \`a $\mathbb{C}\times \mathbb{Z}$.

\noindent Lorsque $\alpha$ est diff\'erent de $1$, on a $\mathrm{C}(\theta_\alpha)=\langle\theta_\alpha^n\,\vert\,n\in\mathbb{Z}\rangle\simeq\mathbb{Z}$.
\end{pro}

\begin{proof}
L'\'eventualit\'e $\alpha=1$ ayant d\'ej\`a \'et\'e trait\'ee nous allons supposer que $\alpha\not=1$.
On v\'erifie, comme on l'a fait pour $\alpha=1$, que la croissance des degr\'es de $\theta_\alpha$ n'est pas born\'ee. Soit $\psi=\left(\frac{A(y)x+B(y)}{C(y)x+D(y)},y+
\tau\right)$ un \'el\'ement de $\mathrm{C}(\theta_\alpha)$. Comme d'habitude on peut supposer que $C$ vaut $0$ ou $1$. 

\noindent Commen\c{c}ons par examiner la possibilit\'e $C=1$; les conditions de commutation impliquent que $$A(y+1)-\alpha A(y)=\frac{\alpha}{y+\tau}.$$ En examinant les 
p\^oles de la diff\'erence $A(y+1)-\alpha A(y)$ on constate que $\tau$ est un entier $n\in\mathbb{Z}$. Quitte \`a composer $\psi$ par~$\theta_\alpha^{-n}$ on se ram\`ene au
 cas $n=0$. Mais si $$A(y+1)-\alpha A(y)=\frac{\alpha}{y}$$ poss\`ede une solution alors $\theta_\alpha$ est conjugu\'e \`a $(\alpha x,y+1)$ dont la croissance des degr\'es 
est born\'ee. 

\noindent Reste \`a examiner la possibilit\'e $\psi=(A(y)x+B(y),y+\tau)$ que l'on rencontre bien s\^ur au travers des it\'er\'es de~$\theta_\alpha$. La commutation de $\psi$ et 
$\theta_\alpha$ indique que $A(y+1)=A(y)$, d'o\`u $A(y)=a\in\mathbb{C}^*$, et produit l'\'equation aux diff\'erences $$B(y+1)-\alpha B(y)=\frac{1}{y+\tau}-\frac{a}{y}.$$
 En examinant les p\^oles des membres de cette \'egalit\'e on constate que $\tau$ est un entier. Par suite en composant par un it\'er\'e ad-hoc de $\theta_\alpha$ on se ram\`ene
 \`a $\psi=(ax+B(y),y)$, $a\in\mathbb{C}^*$, et $B$ v\'erifie $$B(y+1)-\alpha B(y)=\frac{1-a}{y}\,\,\,\,\,(\star).$$ Si $a=1$ on constate que la seule solution rationnelle de 
$(\star)$ est la solution nulle ce qui conduit \`a $\psi=\mathrm{id}$; enfin si $a\not=1$ l'\'equation $(\star)$ ne poss\`ede pas de solution rationnelle, ceci r\'esulte l\`a 
encore d'un examen des p\^oles (ou encore d'un argument d\'ej\`a rencontr\'e de croissance des degr\'es).
\end{proof}

\bigskip

\noindent Reprenons l'exemple suivant \'evoqu\'e dans l'introduction. Soit $(f_{\alpha,\beta})_{\alpha,\beta}$ 
la famille de transformations birationnelles donn\'ee dans la carte affine $z=1$ par 
\begin{align*} 
&\left(\frac{\alpha x+y}{x+1},\beta y\right), && \alpha,\,\beta\in\mathbb{C}^*.
\end{align*} 

\begin{pro}[\cite{De2}, Lemme 1.4, Th\'eor\`eme 1.6]
Pour $\alpha$, $\beta$ g\'en\'eriques, la transformation $f_{\alpha,\beta}$ est \`a croissance lin\'eaire et poss\`ede un centralisateur d\'enombrable; plus pr\'ecis\'ement 
$\mathrm{C}(f_{\alpha,\beta})$ est constitu\'e des puissances de $f_{\alpha,\beta}$.
\end{pro}

\noindent L'id\'ee de la d\'emonstration est la suivante: le point $p=(-1:\alpha:1)$ est \og envoy\'e\fg\, par $f_{\alpha,\beta}$ sur une fibre de la fibration $y=$ cte et 
l'orbite positive de $p$ est constitu\'ee de fibres de cette fibration. Soit $\psi$ une transformation birationnelle qui commute \`a~$f_{\alpha,\beta}$; puisque $\psi$ 
contracte un nombre fini de courbes il existe un entier $k$ positif (que l'on choisit minimal) tel que $f_{\alpha,\beta}^k(p)$ ne soit pas contract\'ee par $\psi$. Quitte 
\`a remplacer $\psi$ par $\widetilde{\psi}:=\psi f_{\alpha,\beta}^{k-1}$ on constate que~$\widetilde{\psi}(p)$ est un point d'ind\'etermination de $f_{\alpha,\beta}$; autrement 
dit $\widetilde{\psi}$ permute les points d'ind\'etermination de $f_{\alpha,\beta}$. Une \'etude plus pr\'ecise permet de montrer que $p$ est fix\'e par $\widetilde{\psi}$. Les 
param\`etres $\alpha$ et $\beta$ \'etant g\'en\'eriques, l'adh\'erence de l'orbite n\'egative de $p$ par $f_{\alpha,\beta}$ est Zariski dense; puisque $\widetilde{\psi}$ fixe 
chaque \'el\'ement de l'orbite de~$p$, $\widetilde{\psi}$ co\"{i}ncide avec l'identit\'e.

\section{R\'ecapitulatif}\label{Sec:res}

\noindent Le tableau qui suit r\'esume les diff\'erents cas rencontr\'es \`a conjugaison birationnelle pr\`es. La colonne \og$\mathrm{C}(\phi)$\fg\, pr\'ecise la suite 
exacte de $\pi_2$ en fonction des propri\'et\'es de $\mathrm{H}$ et $\mathrm{im}\,\pi_2$.

\smallskip

\begin{landscape}
\vspace*{0.8cm}
\begin{small}
\begin{tabular}{|c|c|c|c|c|c|c|}
\hline
   &  &  &  &  && \\
type $\phi$ & type $\pi_2(\phi)$ & $\mathrm{H}=\ker\pi_{2\vert_{\mathrm{C}(\phi)}}$ & $\mathrm{im}\,\pi_{2_{\vert_{\mathrm{C}(\phi)}}}$ &  croissance & $\mathrm{C}(\phi)$ &\\
   &  &  &  & des degr\'es  & &\\
   &  &  &  &  & &\\
   \hline
    &  &  &  &  & &\\
$(x+1,y)$  &$y$ & $\mathrm{J}_a$& $\mathrm{PGL}_2(\mathbb{C})$ &  born\'ee & $\mathrm{J}_a\rtimes\mathrm{PGL}_2(\mathbb{C})$ & $(\ell_1)$\\
    &  &  &  &  & &\\   
       \hline
    &  &  &  &  & &\\
$(ax,y)$, $a\in\mathbb{C}$ & $y$ & $\mathrm{J}_m$ & $\mathrm{PGL}_2(\mathbb{C})$&  born\'ee & $\mathrm{J}_m\rtimes\mathrm{PGL}_2(\mathbb{C})$&$(\ell_2)$\\
    $a$ non racine de l'unit\'e  & & &  &  & &\\
      & & &  &  & &\\
         \hline
    &  &  &  &  & &\\   
$(a(y)x,y)$, $a\in\mathbb{C}(y)\setminus\mathbb{C}$ & $y$ & $\mathrm{J}_m$ & $\mathrm{Stab}(a)$ fini&  lin\'eaire & $\mathrm{J}_m\rtimes\overline{\mathrm{Stab}(a)}$, extension 
finie de $\mathrm{Ab}(\phi)=\mathrm{J}_m$&$(\ell_3)$\\
      & &  &  &  & &\\
         \hline
    &  &  &  &  & &\\   
$\left(\frac{c(y)x+F(y)}{x+c(y)},y\right)$, $\deg F\geq 3$ & $y$ & $\mathrm{J}_F$ & fini &  lin\'eaire &  extension finie de $\mathrm{Ab}(\phi)=\mathrm{J}_F$&$(\ell_4)$\\
     & &  &  &  & &\\
        \hline
    &  &  &  &  & &\\   
 $\left(\frac{c(y)x+y}{x+c(y)},y\right)$ & $y$ & $\mathrm{J}_y$ & $S(c;\alpha)$ &  lin\'eaire & $\mathrm{J}_y\rtimes S(c;\alpha)$, extension finie de $\mathrm{Ab}(\phi)=
\mathrm{J}_y$&$(\ell_5)$\\
  & & &  &  & &\\
     \hline
         &  &  &  & & &\\  
$\phi^k$ d\'ecrit par $(\ell_1)$--$(\ell_5)$ & $\beta y$, $\beta$ d'ordre fini $k$ & & & celle de $\phi^k$ & $\mathrm{C}(\phi)\subset\mathrm{C}(\phi^k)$ d\'ecrit en 
$(\ell_1)$--$(\ell_5)$ & $(\ell_6)$\\ 
  & & &  &  & & \\     
   \hline
    & & &  &  & &\\  
$(x+a(y),\beta y)$  & $\beta y$, $\beta$ d'ordre infini & $\big\{(x+\alpha,y)\,\vert\,\alpha\in\mathbb{C}\big\}$ & $\subset\mathbb{C}^*$ &  lin\'eaire & $0\longrightarrow 
\mathrm{H}\simeq\mathbb{C}\longrightarrow\mathrm{C}(\phi)\longrightarrow\mathrm{im}\,\pi_{2_{\vert_{\mathrm{C}(\phi)}}}\longrightarrow 1$ &$(\ell_7)$\\  
   & & &  &  & & \\
   \hline
    & & &  &  & & \\  
$(a(y)x,\beta y)$  & $\beta y$, $\beta$ d'ordre infini &  $\big\{(\alpha x,y)\,\vert\,\alpha\in\mathbb{C}^*\big\}$ & $\subset\mathbb{C}^*$ &  lin\'eaire &  $0\longrightarrow 
\mathrm{H}\simeq\mathbb{C}^*\longrightarrow\mathrm{C}(\phi)\longrightarrow\mathrm{im}\,\pi_{2_{\vert_{\mathrm{C}(\phi)}}}\longrightarrow 1$ &$(\ell_8)$\\  
   &  &  &  & & &\\
     \hline
    &  &  &  & & &\\  
\og g\'en\'eral multiplicatif\fg & $\beta y$, $\beta$ d'ordre infini & de torsion, fini & $\subset\mathbb{C}^*$ & lin\'eaire & $0\longrightarrow \mathrm{H}\longrightarrow
\mathrm{C}(\phi)\longrightarrow\mathrm{im}\,\pi_{2_{\vert_{\mathrm{C}(\phi)}}}\longrightarrow 1$  &$(\ell_9)$\\ 
     & & &  & & $\mathrm{H}$ fini &  \\  
  & & &  &  & & \\     
 \hline 
\end{tabular}
\end{small}
\end{landscape}

  \begin{landscape}
  \vspace*{3cm}
  \begin{small}
\begin{tabular}{|c|c|c|c|c|c|c|}
\hline
   &  &  &  && &\\
type $\phi$ & type $\pi_2(\phi)$ & $\mathrm{H}=\ker\pi_{2\vert_{\mathrm{C}(\phi)}}$ & $\mathrm{im}\,\pi_{2_{\vert_{\mathrm{C}(\phi)}}}$ &  croissance & $\mathrm{C}(\phi)$ &\\
   & & &  & des degr\'es  & &\\
    &  &  &  &  & &\\
   \hline
   & &  &  &  & &\\   
$(x+a(y),y+1)$ & $y+1$ & $\big\{(x+\beta, y)\,\vert\,\beta\in\mathbb{C}\big\}$& $\subset\mathbb{C}$ &   lin\'eaire & m\'etab\'elien &\\
 $\not\sim(x,y+1)$ &  & & &   & $0\longrightarrow \mathrm{H}\simeq\mathbb{C}\longrightarrow\mathrm{C}(\phi)\longrightarrow\mathrm{im}\,\pi_{2_{\vert_{\mathrm{C}(\phi)}}}
\longrightarrow 1$& $(\ell_{10})$ \\
  & &  &  &  & &\\
     \hline
   & &  &  &  & &\\    
$(b(y)x,y+1)$ &  $y+1$ & $\big\{(\alpha x, y)\,\vert\,\alpha\in\mathbb{C}^*\big\}$ &$\subset\mathbb{C}$ &  lin\'eaire & $0\longrightarrow \mathrm{H}\simeq\mathbb{C}^*
\longrightarrow\mathrm{C}(\phi)\longrightarrow\mathrm{im}\,\pi_{2_{\vert_{\mathrm{C}(\phi)}}}\longrightarrow 1$ & $(\ell_{11})$\\
$\not\sim(x,y+1)$ &  &  & &   & &\\  
  &  &  &  & & &\\  
      \hline 
         &  &  &  && &\\    
 $(b(y)x,y+1)$ & $y+1$ &  $\langle(\beta x,y),\,\left(\frac{B(y)}{x},y\right)\,\vert\, \beta\in\mathbb{C}^*\rangle$ & $\subset\mathbb{C}$ &  lin\'eaire & $0\longrightarrow 
\mathrm{H}\simeq\mathbb{C}^*\rtimes\mathbb{Z}/2\mathbb{Z}\longrightarrow\mathrm{C}(\phi)\longrightarrow\mathrm{im}\,\pi_{2_{\vert_{\mathrm{C}(\phi)}}}\longrightarrow 1$ &\\
  $\not\sim(x,y+1)$ &  & & &  & & $(\ell_{12})$\\ 
  &  &  &  & & &\\  
      \hline 
    &  &  &  & & &\\  
\og g\'en\'eral additif\fg &  $y+1$ & de torsion, fini &  & & $0\longrightarrow \mathrm{H}\longrightarrow\mathrm{C}(\phi)\longrightarrow\mathrm{im}\,\pi_{2_{\vert_{\mathrm{C}
(\phi)}}}\longrightarrow 1$ & $(\ell_{13})$\\ 
     &  & & & & $\mathrm{H}$ fini &\\  
  &  &  &  & & &\\     
   \hline 
\end{tabular}
\end{small}
\end{landscape}

\noindent La ligne $(\ell_5)$ traite les cas $\deg F\leq 2$ (\S \ref{Subsec:crationnelle}).

\noindent Notons que l'on peut pr\'esenter des formes normales lorsque le groupe $\mathrm{H}$ est suffisamment gros. Ainsi en~$(\ell_9)$ et $(\ell_{13})$ il n'y a pas de forme 
normale, tout du moins raisonnable. On constate aussi qu'\`a la ligne $(\ell_9)$ on retrouve la famille $(f_{\alpha,\beta})$ \`a centralisateur isomorphe \`a $\mathbb{Z}$ 
alors que les $\theta_\alpha$, $\alpha\not\in\{0,\,1\}$, dont le centralisateur est aussi $\mathbb{Z}$, sont en $(\ell_{13})$.
Ces deux familles d'exemples montrent que g\'en\'eriquement \`a degr\'e fix\'e le centralisateur d'une transformation birationnelle de $\mathrm{J}$ est isomorphe \`a $\mathbb{Z}$.

\vspace{8mm}

\bibliographystyle{plain}
\bibliography{bibliocroissance}
\nocite{}

\end{document}